\documentclass[a4paper,12pt]{article}
\usepackage{amsmath,amsthm,amsfonts,amssymb,bbm}
\usepackage{graphicx,psfrag,subfigure}
\usepackage{cite}
\usepackage{hyperref}
\usepackage[hmargin=1in,vmargin=1in]{geometry}
\hypersetup{colorlinks=true}

\renewcommand{\d}{\mathrm d}
\newcommand{\op}{\mathrm{op}}
\newcommand{\ext}{\mathrm{ext}}
\newcommand{\R}{\mathbb R}

\newcommand{\wt}{\widetilde}
\newcommand{\wh}{\widehat}
\newcommand{\ol}{\overline}
\renewcommand{\Re}{\operatorname{Re}}

\newcommand{\Ai}{\operatorname{Ai}}
\newcommand{\Herm}{\operatorname{Herm}}
\newcommand{\Res}{\operatorname{Res}}

\newcommand{\id}{\mathbbm{1}}
\renewcommand{\O}{\mathcal{O}}
\renewcommand{\P}{\mathbf P}

\newcommand{\I}{\mathrm i}

\newtheorem{proposition}{Proposition}[section]
\newtheorem{theorem}[proposition]{Theorem}
\newtheorem{lemma}[proposition]{Lemma}

\newtheorem{conjecture}[proposition]{Conjecture}
\theoremstyle{definition}
\newtheorem{remark}[proposition]{Remark}

\numberwithin{equation}{section}

\author{Patrik L.\ Ferrari\thanks{Institute for Applied Mathematics, Bonn University, Endenicher Allee 60, 53115 Bonn,
Germany. E-mail: {\tt ferrari@uni-bonn.de}} \and
B\'alint Vet\H o\thanks{MTA\,--\,BME Stochastics Research Group and Budapest University of Technology and Economics, Egry J.\ u.\ 1, 1111 Budapest, Hungary. E-mail: {\tt vetob@math.bme.hu}}}

\title{The hard-edge tacnode process for Brownian motion}
\date{}

\begin{document}

\maketitle
\sloppy

\begin{abstract}
We consider $N$ non-intersecting Brownian bridges conditioned to stay below a fixed threshold.
We consider a scaling limit where the limit shape is tangential to the threshold.
In the large $N$ limit, we determine the limiting distribution of the top Brownian bridge conditioned to stay below a function as well as the limiting correlation kernel of the system.
It is a one-parameter family of processes which depends on the tuning of the threshold position on the natural fluctuation scale.
We also discuss the relation to the six-vertex model and to the Aztec diamond on restricted domains.
\end{abstract}

\section{Introduction}
Non-intersecting walks have appeared naturally in the descriptions of many physical systems as well as in mathematics.
To mention just a few examples, the polynuclear growth model (describing the growth of an interface) is based on the representation as non-intersecting random walks~\cite{PS02,Jo01},
the Aztec diamond (and similar combinatorial models of random tiling) has a similar mathematical description~\cite{Jo03,BF15},
Markov chains on Young diagrams related to the Plancherel measure~\cite{BO04,Bor11},
and the evolution of eigenvalues of random matrices as the GUE Dyson's Brownian motion~\cite{Dys62} can be expressed and analyzed as non-intersecting Brownian motions~\cite{EM97,FN98}.
The analysis was possible because of the determinantal structure of correlation functions~\cite{EM97,Bor98,RB04}.

In this paper, we study non-intersecting Brownian motions starting and ending at a fixed position with the extra constraint that they stay below a given threshold as illustrated in Figure~\ref{FigTacnode}.
\begin{figure}
\begin{center}
\includegraphics[height=6cm]{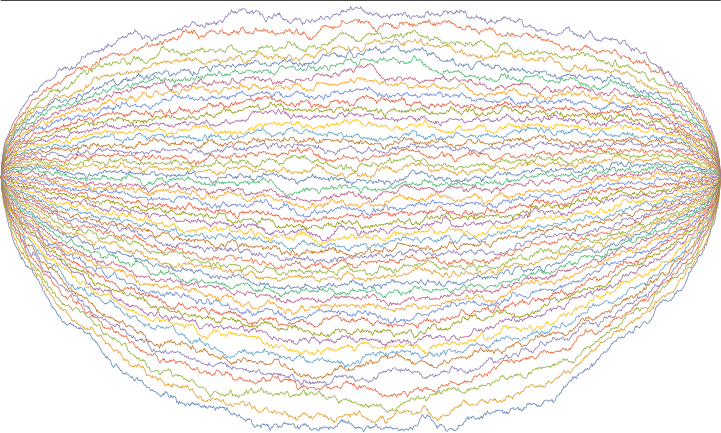}
\caption{Illustration of $N=50$ non-intersecting Brownian motions conditioned to stay below the black threshold.}
\label{FigTacnode}
\end{center}
\end{figure}
The motivation for these investigations is twofold:

\smallskip
(a) The six-vertex model with domain wall boundary conditions (DWBC) can be expressed as a system of non-intersecting line ensembles
(in discrete space and time)~\cite{FS06} with fixed starting and ending points.
In particular, at the free-fermion line, there is a mapping to the Aztec diamond~\cite{ZJ00} and thus by~\cite{Jo03}
we know that the border of the lines are described in the limit of large system by the Airy$_2$ process~\cite{PS02}.
Recent studies of limit shapes (not only for the free-fermion case) consider also geometries beyond the classical DWBC~\cite{CP13b,CP15,CPS16,CPS16b}.
This raised the natural question on the description of the limit process for the border of the line ensemble for $L$-shaped domains or
for pentagonal domains obtained from a square by removing a triangle at the corner.
Although we do not do the analysis for this discrete case, if the removed triangular piece is tangential to the the limit shape of the lines for the DWBC,
then under appropriate scaling, the limiting process should be exactly the one we study in this paper.
See Section~\ref{s:6vertex} for further discussions.

\smallskip
(b) Non-intersecting Brownian motions have attracted a lot of interest also because of their relations to the eigenvalues of Hermitian random matrices
subjected to Dyson's Brownian motion~\cite{AvMD08,TW03b,FN98,BFPSW09,KT04,KT07,TW06,KT07b,BK06}.
Discrete versions have been studied as \mbox{well~\cite{NKT02,Nag03,Jo02b,PS02,FS03}}.
More recently, the situation where the limit shape of two sets of non-intersecting Brownian motions just touch in a tacnode geometry has been studied,
first in a random walk setting in~\cite{FAvM10}, then via a $4\times 4$ Riemann--Hilbert problem~\cite{DKZ10} and with a more direct approach in~\cite{Joh10,FV11}.
The equality of the formulas for the correlation kernel of the tacnode process obtained in~\cite{DKZ10} and in~\cite{FV11} was verified directly in~\cite{D12}.
The tacnode was observed also in random tiling models~\cite{AJvM14,ACJvM13}.

The tacnode geometry occurs also if the non-intersecting trajectories are conditioned to stay positive and to start and end away from $0$ at a distance so that the limit shape becomes tangential to $0$.
This has been studied in~\cite{D13,DV14} for the case of non-intersecting squared Bessel processes.
Since Brownian motion conditioned to stay positive is a Bessel process of parameter $1/2$,
the kernel of the hard-edge tacnode process for non-intersecting Brownian motions can be obtained from~\cite{D13} in terms of the solution of a $4\times4$ Riemann--Hilbert problem.
However, finding explicit formulas for the tacnode limit process for Brownian motions conditioned to stay positive remained open
due to the fact that the hard-edge tacnode kernel was found in~\cite{DV14} explicitly only for non-intersecting squared Bessel processes with integer parameter.

We mention that if the starting points of Brownian bridges (or more generally Bessel processes) are set to $0$, the ending points are the same for all paths and it is scaled with the number of paths,
then the limit shape of the non-intersecting paths conditioned to stay positive separates from $0$ at some time in $(0,1)$.
In the neighbourhood of the point of separation, the hard-edge Pearcey process appears~\cite{DF08,KMW14,DV14}.

In this paper, we consider $N$ Brownian bridges starting from $0$ at time $0$ and ending at $0$ at time $1$.
We condition the Brownian bridges not to intersect for times $t\in (0,1)$ and denote by $B_N(t)$ the position of the top bridge at time $t$.
This is also known as Brownian watermelon and it is well-known that under appropriate scaling, $B_N$ converges to the Airy$_2$ process ${\mathcal A}_2$:
\begin{equation}
2N^{1/6}\left(B_N\left(\tfrac12(1+u N^{-1/3})\right)-\sqrt{N}\right)\to {\mathcal A}_2(u)-u^2
\end{equation}
as $N\to\infty$.
Therefore, if we consider the Brownian watermelon conditioned to stay below a threshold of height $\sqrt{N}+\tfrac12 R N^{-1/6}$,
then the probability that the conditioning is effective is in $(0,1)$ also in the $N\to\infty$ limit.
Thus we will see a new non-trivial limit process which we call \emph{hard-edge tacnode process for Brownian motions}.
This process is characterized by its finite dimensional distributions as given in Theorem~\ref{thm:asymptotics}.
When $R\to\infty$, the constraint becomes irrelevant and the top path will be the Airy$_2$ process (see the discussion after Theorem~\ref{thm:asymptoticsB}).
When $R\to-\infty$, after appropriate rescaling, the limit process should be the one with extended Bessel kernel~\cite{TW03b}
which was also derived for non-intersecting Brownian excursions studied in~\cite{TW07b}.

The derivation of our result does not use the standard determinantal point process approach~\cite{Jo05},
rather we start with a Fredholm determinant expression with path integral kernel obtained in~\cite{NR15,BCR13}
which gives the probability that the top path of $N$ Brownian bridges stays below a given function over an open time subinterval of $[0,1]$, see Proposition~\ref{prop:NguyenRemenik}.
First we extend the conditioning to the full time interval (see Theorem~\ref{thm:NBBbelowh}).
The finite dimensional distributions are then written as ratios of probabilities for two threshold functions leading to Theorem~\ref{thm:DistrH}.
Using~\cite{BCR13}, we can rewrite the Fredholm determinant of a path integral kernel to a Fredholm determinant of an extended kernel
which is indeed the correlation kernel as shown in Theorem~\ref{thm:corrkernel}.
(\hspace{-0.1em}\cite{BCR13} is a generalization of what was present in~\cite{PS02}.
The importance of~\cite{PS02} was rediscovered and extended in~\cite{CQR11} in the setting of the Airy processes.)
Notice that with the present method, we directly get formulas for quantities such as distribution of the maximum of $B_N$ (conditioned to stay below the threshold).
This quantity is not directly accessible by the standard method leading to the finite dimensional distributions.
Finally we perform the asymptotic analysis for the correlation kernel (see Theorem~\ref{thm:asymptotics})
and we give the limit of the probability that the top path of the non-intersecting Brownian bridges stays below a rescaled function (see Theorem~\ref{thm:asymptoticsB}).

After the appearance of the first version of this paper, non-intersecting Brownian bridges with reflecting and absorbing walls were studied in~\cite{LW17} by the method of orthogonal polynomials.
Their correlation kernel of the hard-edge tacnode process by the solution of a $2\times2$ Riemann--Hilbert problem both for reflecting and absorbing walls is less explicit than our formulation.
In a second step they show that the kernel $\wh K^{\ext}(T_1,U_1;T_2,U_2)$ in Theorem~\ref{thm:asymptotics} below is the odd part of the soft-edge tacnode process of~\cite{FV11}
thus proving the equivalence of the kernel $\wh K^{\ext}(T_1,U_1;T_2,U_2)$ and their formula for the hard-edge tacnode process in the case of absorbing walls.

\paragraph{Outline:}
In Section~\ref{s:results}, we define the model and present the results of this paper.
Section~\ref{s:6vertex} contains a short discussion on the relation with the six-vertex model.
In Section~\ref{s:mp}, we determine the multipoint distribution of $B_N$ conditioned to stay below a constant threshold.
Section~\ref{s:extension} contains the extension of the formula of~\cite{NR15} to the full time interval.
In Section~\ref{s:directcorr}, we prove the formula for the correlation kernel.
The large $N$ asymptotic analysis is performed in Section~\ref{s:asymptotics}.
Finally, Section~\ref{s:lemmas} contains the proof of several technical lemmas.

\paragraph{Acknowledgements:}
The authors are grateful for discussions with F.\ Colomo and A.\ Sportiello about their work and to both ICERM and the Galileo Galilei Institute
which provided the platform to make such discussions possible.
The work of P.L.~Ferrari is supported by the German Research Foundation via the SFB 1060--B04 project.
The work of B.\ Vet\H o is supported by OTKA (Hungarian National Research Fund) grant K100473.
His work is supported by the \'UNKP--16--4--III.\ New National Excellence Program of the Ministry of Human Capacities.
He is grateful for the Postdoctoral Fellowship of the Hungarian Academy of Sciences and for the Bolyai Research Scholarship.

\section{Model and main results}\label{s:results}

\subsubsection*{The model}
The model considered in this paper is the following system of $N$ non-intersecting Brownian bridges.
Consider $N$ standard Brownian bridges $B_1(t),\dots,B_N(t)$ which start from zero at time $t=0$ and end at zero at time $t=1$, and condition them on having no intersection in $t\in(0,1)$ in Doob's sense.
To denote the paths, we use the convention \mbox{$B_1(t)\le\dots\le B_N(t)$} with strict inequality for $t\in(0,1)$.

The starting point of the work is a formula for the distribution of the top path $B_N(t)$ conditioned to stay below a given function, based on \cite{BCR13} and \cite{NR15}.
To state it, we need some notations.
Let $H_n(x)$ denote the $n$th Hermite polynomial defined by
\begin{equation}
H_n(x)=(-1)^ne^{x^2}\frac{\d^n}{\d x^n}e^{-x^2}
\end{equation}
which form an orthogonal system with respect to the weight $e^{-x^2}\d x$ on $\R$, i.e.
\begin{equation}
\int_\R H_n(x)H_m(x)e^{-x^2}\d x=\sqrt\pi2^nn!\delta_{n,m}.
\end{equation}
Define the harmonic oscillator functions
\begin{equation}\label{defphi}
\varphi_n(x)=\pi^{-1/4}2^{-n/2}(n!)^{-1/2}e^{-x^2/2}H_n(x)
\end{equation}
and the Hermite kernel
\begin{equation}\label{defKHerm}
K_{\Herm,N}(x,y)=\sum_{n=0}^{N-1}\varphi_n(x)\varphi_n(y).
\end{equation}
With the Laplacian $\Delta$ on $\R$, let
\begin{equation}\label{defD}
D=-\frac12(\Delta-x^2+1)
\end{equation}
be the differential operator for which the eigenfunctions are the harmonic oscillator functions, that is, $D\varphi_n=n\varphi_n$.
Then $K_{\Herm,N}$ is a projection to the space spanned by the eigenfunctions $\varphi_0,\dots,\varphi_{N-1}$.

For some $0<a<b<1$, let $H^1([a,b])$ be the set of square integrable functions with square integrable derivative.
The following statement is a consequence of Propositions~2.1 (which goes back to Proposition~4.3 of~\cite{BCR13}) and Proposition~2.2 in~\cite{NR15}.
\begin{proposition}[Nguyen-Remenik~\cite{NR15}]\label{prop:NguyenRemenik}
Let $0<a<b<1$ and $h\in H^1([a,b])$ and denote by $B_N(t)$ the top path of $N$ non-intersecting Brownian bridges.
Then
\begin{equation}\label{condbelowh}
\P\left(B_N(t)<h(t)\mbox{ for }t\in[a,b]\right)=\det\left(\id-K_{\Herm,N}+\Theta_{A,B}e^{(B-A)D}K_{\Herm,N}\right)_{L^2(\R)}
\end{equation}
where
$A=\frac12\ln\frac a{1-a}$, $B=\frac12\ln\frac b{1-b}$, and $D$ is the differential operator defined in \eqref{defD}.
Further,
\begin{multline}\label{defTheta}
\Theta_{A,B}(x,y)=e^{(y^2-x^2)/2+B}\frac{\exp\left(-\frac{(e^By-e^Ax)^2}{4(\beta-\alpha)}\right)}{\sqrt{4\pi(\beta-\alpha)}}\\
\times\P_{\wh b(\alpha)=e^Ax,\wh b(\beta)=e^By}\left(\wh b(\tau)\le\frac{1+4\tau}{\sqrt2}h\left(\frac{4\tau}{1+4\tau}\right)\mbox{ for }\tau\in[\alpha,\beta]\right)
\end{multline}
where $\alpha=\frac14e^{2A}=\frac14\frac a{1-a}$ and $\beta=\frac14e^{2B}=\frac14\frac b{1-b}$.
In \eqref{defTheta}, $\wh b(\tau)$ denotes a Brownian bridge with diffusion coefficient $2$ starting at $\wh b(\alpha)=e^Ax$ and ending at $\wh b(\beta)=e^By$.
\end{proposition}

\subsubsection*{Finite $N$ result}
First of all, we extend Proposition~\ref{prop:NguyenRemenik} so that the condition for the $N$ non-intersecting Brownian bridges to stay below a function can be imposed for the whole $[0,1]$.
Since we are ultimately interested in the distribution of $N$ non-intersecting Brownian bridges conditioned to stay below a constant,
we consider functions $h$ such that for some $0<t_1<t_2<1$ and $r>0$,
\begin{equation}\label{defhr}
h(t)\le r\quad\mbox{for}\quad t\in[0,1]\qquad\mbox{and}\qquad h(t)=r\quad\mbox{for}\quad t\in[0,1]\setminus(t_1,t_2).
\end{equation}
Motivated by the definition \eqref{defTheta}, let
\begin{equation}\label{deftauhtilde}
\tau_i=\frac14\frac{t_i}{1-t_i}\quad\mbox{for}\quad i=1,2\qquad\mbox{and}\qquad\wt h(\tau)=\frac{1+4\tau}{\sqrt2}\left[h\left(\frac{4\tau}{1+4\tau}\right)-r\right].
\end{equation}
Further, for such a function $h$, define
\begin{multline}\label{defT1}
T_{\alpha_1,\alpha_2}^h(u,v)\\
=\frac{\d}{\d v} \P_{\wt b(\alpha_1)=u}\left(\wt b(\tau)\le0\mbox{ for }\tau\in[\alpha_1,\alpha_2],\wt b(\tau)\le\wt h(\tau)\mbox{ for }\tau\in(\tau_1,\tau_2),\wt b(\alpha_2)\leq v\right)
\end{multline}
where $\tau_1,\tau_2\in [\alpha_1,\alpha_2]$ and $\wt h$ are as in \eqref{deftauhtilde}. The Brownian motion $\wt b$ above has diffusion coefficient $2$.

For any $u,v\in\R$ and $n,m$ integers, introduce the functions
\begin{align}
\Phi_\tau^n(u)&=\frac1{\pi\I}\int_{\I\R}\d W\,W^ne^{\tau\left(\sqrt2r-2W\right)^2-\sqrt2rW}(f_W(u)-f_W(-u)),\label{defPhi}\\
\Psi_\tau^m(v)&=\frac1{2\pi\I}\oint_{\Gamma_0}\d Z\,Z^{-(m+1)}e^{-\tau\left(\sqrt2r-2Z\right)^2+\sqrt2rZ}(g_Z(v)-g_Z(-v))\label{defPsi}
\end{align}
with
\begin{equation}\label{deffg}
f_W(u)=e^{(\sqrt2r-2W)u}\quad\mbox{and}\quad g_Z(v)=e^{-(\sqrt2r-2Z)v}
\end{equation}
and define the kernel
\begin{equation}\label{defK0}
K_0(n,m)=\frac1{2\pi\I}\oint_{\Gamma_0}\d Z\,\frac{(\sqrt2r-Z)^n}{Z^{m+1}}e^{-2r^2+2\sqrt2rZ}.
\end{equation}
They satisfy the following compatibility conditions (see Section~\ref{s:lemmas} for the proof).
\begin{proposition}\label{prop:compatibility}
Let $\phi_t(x,y)=\frac{1}{\sqrt{2\pi t}} \exp\left(-(x-y)^2/2t\right)$ and set
\begin{equation}\label{defT2}
T_{\tau_1,\tau_2}(x,y)=\phi_{2(\tau_2-\tau_1)}(y-x)-\phi_{2(\tau_2-\tau_1)}(y+x)
\end{equation}
for any $x,y\in\R$. Then, for any $0<\tau_1<\tau_2$, $0<\tau$ and $u,v\in\R$, the following compatibility relations are satisfied:
\begin{align}
\int_{\R_-}\d u\,\Phi_{\tau_1}^n(u)T_{\tau_1,\tau_2}(u,v)&=\Phi_{\tau_2}^n(v),\label{PhiT}\\
\int_{\R_-}\d v\,T_{\tau_1,\tau_2}(u,v)\Psi_{\tau_2}^m(v)&=\Psi_{\tau_1}^m(u),\label{TPsi}\\
\int_{\R_-}\d u\,\Phi_\tau^n(u)\Psi_\tau^m(u)&=(\id-K_0)(n,m).\label{PhiPsi}
\end{align}
\end{proposition}

We can now state the extension of Proposition~\ref{prop:NguyenRemenik} to the conditioning on the full time interval.
\begin{theorem}[Full time span conditioning]\label{thm:NBBbelowh}
Let the function $h\in H^1([0,1])$ satisfy~\eqref{defhr} for some $0<t_1<t_2<1$. Then
\begin{equation}\label{NBBbelowh}
\P(B_N(t)<h(t)\mbox{ for }t\in[0,1])=\det\left(\id-K_N^h\right)_{L^2(\{0,1,\dots,N-1\})}
\end{equation}
where the kernel $K_N^h$ is given by
\begin{equation}\label{defKN}
K_N^h(n,m)=\id(n,m)-\int_\R\d u\int_\R\d v\,\Phi_{\tau_1}^n(u)T_{\tau_1,\tau_2}^h(u,v)\Psi_{\tau_2}^m(v).
\end{equation}
\end{theorem}

As a consequence, we get the following for the probability that the conditioned process remains below a given function.
\begin{theorem}\label{thm:DistrH}
Under the assumptions of Theorem~\ref{thm:NBBbelowh}, we have
\begin{multline}\label{NBBbelowhCond}
\P(B_N(t)<h(t)\mbox{ for }t\in[0,1]\bigm|B_N(t)<r\mbox{ for }t\in[0,1])\\
=\det\left(\id-K_{\tau_1}+T^h_{\tau_1,\tau_2}K_{\tau_2,\tau_1}\right)_{L^2(\R_-)}
\end{multline}
where $K_{\tau_1}=K_{\tau_1,\tau_1}$ and $K_{\tau_2,\tau_1}$ is given by
\begin{equation}\label{defKtau}
K_{\tau_2,\tau_1}(u,v)=\sum_{n,m=0}^{N-1}\Psi_{\tau_2}^n(u)\,(\id-K_0)^{-1}(n,m)\,\Phi_{\tau_1}^m(v).
\end{equation}
\end{theorem}

For $N$ non-intersecting Brownian bridges conditioned to stay below a constant level $r$ for $[0,1]$,
we know by the Karlin--McGregor type formulas and Eynard--Mehta \mbox{theorem~\cite{EM97,KM59}} that it forms a determinantal process.
We compute its correlation kernel which characterizes the finite dimensional distributions of the process.
Conditioning $N$ non-intersecting Brownian bridges to stay below $r$ corresponds to the $h\equiv r$ constant choice in \eqref{defhr}.
In this case, \eqref{defT1} becomes $\id_{u<0}T_{\alpha_1,\alpha_2}(u,v)\id_{v<0}$ by the reflection principle.
The correlation kernel of $N$ non-intersecting Brownian bridges conditioned to be below a constant level is given as follows.

\begin{theorem}[Correlation kernel]\label{thm:corrkernel}
The system of $N$ non-intersecting Brownian bridges conditioned to stay below the constant level $r$ for time $[0,1]$
forms a determinantal process with extended correlation kernel defined for $t_1,t_2\in[0,1]$ and $x_1,x_2\le r$ by
\begin{equation}\label{defKexttx}
K_\ext(t_1,x_1;t_2,x_2)=\frac1{\sqrt{2(1-t_1)(1-t_2)}}K^\ext(\tau_1,u_1;\tau_2,u_2)
\end{equation}
where we used the variables
\begin{equation}\label{deftauu}
\tau_i=\frac14\frac{t_i}{1-t_i},\qquad u_i=\frac{x_i-r}{\sqrt2(1-t_i)}
\end{equation}
due to \eqref{deftauhtilde} and the kernel
\begin{equation}\label{defKexttauu}
K^\ext(\tau_1,u_1;\tau_2,u_2)=-\id_{\tau_1<\tau_2}T_{\tau_1,\tau_2}(u_1,u_2)+\sum_{n,m=0}^{N-1}\Psi_{\tau_1}^n(u_1)(\id-K_0)^{-1}(n,m)\Phi_{\tau_2}^m(u_2).
\end{equation}
In particular, the gap probabilities of $N$ non-intersecting Brownian bridges conditioned to stay below level $r$ can be expressed for any $t_1,\dots,t_k\in[0,1]$ and $h_1,\dots,h_k\le r$ as
\begin{equation}\label{BNgapprob}
\P\left(B_N(t_1)<h_1,\dots,B_N(t_k)<h_k\bigm|B_N(t)<r, t\in[0,1]\right)
=\det(\id-QK^\ext)_{L^2(\{\tau_1,\dots,\tau_k\}\times\R_-)}
\end{equation}
with
\begin{equation}\label{defQeta}
Qf(\tau_i,u)=\id_{u\ge\eta_i}f(\tau_i,u),\qquad\tau_i=\frac14\frac{t_i}{1-t_i},\qquad\eta_i=\frac{1+4\tau_i}{\sqrt2}(h_i-r).
\end{equation}
\end{theorem}

\subsubsection*{Large $N$ asymptotic result}
Next we take the number of Brownian paths $N\to\infty$.
We choose the scaling in a way that the following condition holds.
The probability that $N$ non-intersecting Brownian bridges stay below the rescaled threshold $r$ should stay asymptotically away from $0$ and $1$.
This means that we need to scale the threshold $r$ as well as time and space as follows:
\begin{equation}
t=\frac{1+TN^{-1/3}}{2},\quad r=\sqrt N+\frac{RN^{-1/6}}{2},\quad h=\sqrt N+\frac{(R+H) N^{-1/6}}{2}
\end{equation}
with $H\leq 0$. Let us first describe ingredients of the limiting correlation kernel.
For any parameter $s$, let
\begin{equation}
\Ai^{(s)}(x)=e^{2s^3/3+xs}\Ai(s^2+x).
\end{equation}
Then we introduce the functions
\begin{equation}\label{defPhihat}
\begin{aligned}
\wh\Phi_T^\xi(U)&=\Ai^{(T)}(R+\xi+U)-\Ai^{(T)}(R+\xi-U),\\
\wh\Psi_T^\zeta(U)&=\Ai^{(-T)}(R+\zeta+U)-\Ai^{(-T)}(R+\zeta-U),
\end{aligned}
\end{equation}
and the shifted GOE kernel
\begin{equation}
\wh K_0(\xi,\zeta)=2^{-1/3}\Ai(2^{-1/3}(2R+\xi+\zeta)).\label{defK0hat}
\end{equation}
The next theorem establishes the convergence of the rescaled kernel and the existence of the hard-edge tacnode process which is the limiting determinantal point process.

\begin{theorem}[The hard-edge tacnode process]\label{thm:asymptotics}
Consider the scaling
\begin{equation}\label{scaling}
t_i=\frac{1+T_iN^{-1/3}}2,\quad r=\sqrt N+\frac{RN^{-1/6}}2,\quad x_i=\sqrt N+\frac{(R+U_i)N^{-1/6}}2.
\end{equation}
Then the extended correlation kernel of $N$ non-intersecting Brownian bridges conditioned to stay below a constant level converges uniformly on compact sets, i.e.
\begin{equation}\label{Kextconv}
\lim_{N\to\infty}\frac{N^{-1/6}}2K_\ext(t_1,x_1;t_2,x_2)=\wh K^\ext(T_1,U_1;T_2,U_2)
\end{equation}
where the limiting kernel $\wh K^\ext$ is given by
\begin{equation}\label{defKexthat}
\wh K^\ext(T_1,U_1;T_2,U_2)=-\id_{T_1<T_2}T_{T_1,T_2}(U_1,U_2)+\int_{\R_+}\!\!\d\xi\int_{\R_+}\!\!\d\zeta\,\wh\Psi_{T_1}^\xi(U_1)(\id-\wh K_0)^{-1}(\xi,\zeta)\wh\Phi_{T_2}^\zeta(U_2)
\end{equation}
where $T_1,T_2\in\R$ and $U_1,U_2\le0$.

As a consequence, the \emph{hard-edge tacnode process $\mathcal T$} exists as the limit of $N$ non-intersecting Brownian bridges conditioned to stay below a constant level under the given scaling.
It is characterized by the following gap probabilities.
For any fixed integer $k$ and $T_1,\dots,T_k\in\R$ and for any compact set $E\subseteq\{T_1,\dots,T_k\}\times\R_-$,
\begin{equation}
\P(\mathcal T\cap E=\emptyset)=\det\big(\id-\wh K^\ext\big)_{L^2(E)}.
\end{equation}
\end{theorem}

As in~\cite{D12} and in~\cite{DV14}, the soft-edge or hard-edge tacnode process usually has a natural temperature parameter (here is the threshold $R$),
and the derivative of the correlation kernel with respect to the temperature parameter has a low rank structure.
In particular, the temperature derivative of the correlation kernel of the soft-edge tacnode process is rank two,
which was proved in~\cite{D12} to hold for the formulas obtained in~\cite{DKZ10} and in~\cite{FV11} yielding a direct proof for the equivalence of the two formulation.
In~\cite{DV14}, the rank one structure of the temperature derivative of the hard-edge tacnode kernel was shown in the case of non-intersecting squared Bessel processes with integer parameter.
This gives the importance of the next proposition about the derivative with respect to the microscopic position parameter of the threshold since the model studied in the present paper
corresponds to non-intersecting Bessel processes of parameter $1/2$.
The proposition is proved in Section~\ref{s:lemmas}.

\begin{proposition}\label{prop:derivative}
The derivative of the extended correlation kernel of the hard-edge tacnode process with respect to parameter $R$ has rank one, that is,
\begin{equation}
\frac\partial{\partial R}\wh K^\ext(T_1,U_1;T_2,U_2)=-f(T_1,U_1)g(T_2,U_2)
\end{equation}
where
\begin{align}
f(T_1,U_1)&=\int_{\R_+}\d\xi\,\wh\Psi_{T_1}^\xi(U_1)(\id-\wh K_0)^{-1}(\xi,0),\\
g(T_2,U_2)&=\int_{\R_+}\d\zeta\,(\id-\wh K_0)^{-1}(0,\zeta)\wh\Phi_{T_2}^\zeta(U_2).
\end{align}
\end{proposition}

In Proposition 1.5 of~\cite{LW17}, it was proved that the odd part of the soft-edge tacnode kernel of~\cite{FV11} coincides
with the extended correlation kernel of the hard-edge tacnode process defined by \eqref{defKexthat}.
We recall the statement below.

\begin{proposition}[Proposition 1.5 of~\cite{LW17}]\label{prop:liechtywang}
Let $\mathcal L^{\lambda,\sigma}_{\rm tac}(T_1,U_1,T_2,U_2)$ be correlation kernel of the soft-edge tacnode process as defined in (1.5) of~\cite{FV11}
where $\lambda$ is the asymmetry parameter and where $\sigma$ is the temperature parameter. The symmetric case corresponds to $\lambda=1$.
Then for any threshold $R\in\R$ for the hard-edge tacnode process,
\begin{equation}
\wh K^{\ext}(T_1,U_1;T_2,U_2)=\mathcal L^{1,2^{2/3}R}_{\rm tac}(T_1,U_1,T_2,U_2)-\mathcal L^{1,2^{2/3}R}_{\rm tac}(T_1,U_1,T_2,-U_2)
\end{equation}
holds.
\end{proposition}

\begin{remark}
It is possible to view the system of non-intersecting Brownian bridges conditioned to stay below a constant threshold $r$
as non-intersecting paths $r-Y_t^{(i)}$ for $i=1,2,\dots,N$ where $Y_t^{(i)}$ are three-dimensional Bessel bridges, hence the results of~\cite{D13} apply.
Since the correlation kernel in~\cite{D13} is expressed with the solution of a $4\times4$ Riemann--Hilbert problem, it is very hard to compare the two kernels.
Proposition~\ref{prop:derivative} is the first step towards this aim.
Although this approach was successful for the soft-edge tacnode process (see~\cite{D12}), the hard-edge tacnode case seems to be more difficult
and the results of~\cite{DV14} and of~\cite{D13} could not be compared so far.
\end{remark}

Theorem~\ref{thm:asymptotics} characterizes the finite dimensional distributions of the limit process,
which does not cover properties such as the limiting probability that the non-intersecting paths stay below a given function.
This can be obtained by performing the large $N$ asymptotics of Theorem~\ref{thm:DistrH}.

\begin{theorem}\label{thm:asymptoticsB}
Consider the top path of $N$ non-intersecting Brownian motions conditioned to stay below $r=\sqrt N+\tfrac12 RN^{-1/6}$ rescaled as
\begin{equation}
{\mathcal B}^R_N(T)=2 N^{1/6}\left( B_N\left(\tfrac12(1+T N^{-1/3})\right)-\sqrt{N}\right).
\end{equation}
Let $T_1<T_2$ be given as well as a function $H\in H^1([T_1,T_2])$ with $H\leq R$.
Then
\begin{equation}\label{eq2.38}
\lim_{N\to\infty} \P\left({\mathcal B}^R_N(T)\leq H(T)\mbox{ for }T\in[T_1,T_2]\right)=\det(\id - \wh K_{T_1}+\wh T^{H-R}_{T_1,T_2}\wh K_{T_2,T_1})_{L^2(\R_-)}
\end{equation}
where $\wh K_{T_1}=\wh K_{T_1,T_1}$ and $\wh K_{T_1,T_2}(U_1,U_2):=\wh K^\ext(T_1,U_1;T_2,U_2)$ defined in \eqref{defKexthat} and
\begin{equation}
\wh T^H_{T_1,T_2}(U_1,U_2)=\frac{\d}{\d U_2} \P_{B(T_1)=U_1}\left(B(T)\le H(T)\mbox{ for }T\in[T_1,T_2],B(T_2)\leq U_2\right)
\end{equation}
with $B(T)$ being a Brownian motion with diffusion coefficient $2$.
\end{theorem}

Finally, let us discuss the large $R$ limit. As $R\to\infty$, the constraint $H\leq R$ becomes trivially satisfied and thus we should recover
\begin{equation}\label{ew2.39}
\lim_{R\to\infty} \det(\id - \wh K_{T_1}+\wh T^{H-R}_{T_1,T_2}\wh K_{T_2,T_1})_{L^2(\R_-)} =\P({\cal A}_2(T)-T^2\leq H(T),T\in[T_1,T_2])
\end{equation}
where ${\cal A}_2$ is the Airy$_2$ process.
We verify it below.
Consider the entry $(U,U')$ of the kernel on the left-hand side of \eqref{ew2.39}.
Applying the change of variables $U\to U-R$ and $U'\to U'-R$, we have that the right-hand side of \eqref{eq2.38} is the Fredholm determinant on $L^2((-\infty,R))$ with kernel
\begin{equation}
- \wh K_{T_1}(U-R,U'-R)+\int\d V \wh T^{H-R}_{T_1,T_2}(U-R,V-R) \wh K_{T_2,T_1}(V-R,U'-R).
\end{equation}
It is easy to verify that
\begin{equation}\label{eqAiryKernel}\begin{aligned}
&\lim_{R\to\infty} \wh K^\ext(T,U-R;T',U'-R)\,e^{2(T^3-{T'}^3)/3+T U-T'U'}\\
&\qquad= K_{\Ai}(T,U+T^2;T',U'+{T'}^2)\\
&\qquad=-\frac{e^{-\frac{(U-U')^2}{4 (T'-T)}}}{\sqrt{4\pi (T'-T)}}\id_{T<T'} + \int_{\R_+} \d\xi\,e^{\xi(T'-T)}\Ai(\xi+U+T^2)\Ai(\xi+U'+{T'}^2)
\end{aligned}\end{equation}
where $K_{\Ai}$ is know as the extended Airy kernel~\cite{PS02,Jo03}.
Indeed, as $R\to\infty$, $(\id-\wh K_0)^{-1}\to \id$, but also $\wh\Phi_T^\xi(U-R) e^{-2 T^3/3-T U}\to e^{T \xi}\Ai(\xi+U+T^2)$, and $\wh\Psi_T^\zeta(U-R) e^{2 T^3/3+T U}\to e^{-T \zeta}\Ai(\zeta+U+T^2)$.

These asymptotics imply that in the $R\to\infty$ limit our Fredholm determinant is on $L^2(\R)$ with kernel
\begin{equation}
- K_{\rm Ai}(T_1,U+T_1^2;T_1,U'+T_1^2)+\int\d V \frac{e^{\frac23 T_1^3+T_1 U}}{e^{\frac23 T_2^3+T_2 V}}\wh T^{H}_{T_1,T_2}(U,V) K_{\rm Ai}(T_2,V+T_2^2;T_1,U'+T_1^2)
\end{equation}
after the same conjugation as in \eqref{eqAiryKernel}.
Finally, by the change of variables $U\to U-T_1^2$, $U'\to U'-T_1^2$, and $V\to V-T_2^2$, the kernel becomes
\begin{equation}
- K_{\rm Ai}(T_1,U;T_1,U')+\int\d V \frac{e^{-\frac13 T_1^3+T_1 U}}{e^{-\frac13 T_2^3+T_2 V}} \wh T^{H}_{T_1,T_2}(U-T_1^2,V-T_2^2) K_{\rm Ai}(T_2,V;T_1,U').
\end{equation}
By Theorem~2 and~3 of~\cite{CQR11}, the determinant of this kernel on $L^2(\R)$ is equal to the the right-hand side of \eqref{ew2.39} as expected\footnote{In Theorem~3 of~\cite{CQR11}
there is a misprint: in the Gaussian factor, one should replace $x$ by $x-\ell^2$ and $y$ by $y-r^2$, as it can be easily verified by comparing with formula preceding Theorem~3.}.

\begin{remark}
There are two natural ways to obtain the hard-edge tacnode process as the limit of non-intersecting Brownian bridges conditioned to stay below a constant level.
The first option is what we follow in the present paper:
we keep the number of paths fixed first and we characterize the distribution of the paths conditioned to stay below a constant level for $[0,1]$, see Theorem~\ref{thm:corrkernel}.
Then we let the number of paths $N\to\infty$ as it is done in Theorem~\ref{thm:asymptotics}.

An alternative approach is that one imposes the condition that the Brownian bridges stay under a constant level on a fixed interval $[a,b]$ with $0<a<b<1$ and one lets the number of paths $N\to\infty$ first.
Then the limit is an Airy$_2$ process conditioned to stay below a parabola for a fixed finite interval (compare with \eqref{ew2.39} for constant $H$).
In the second step, by letting this interval grow to $\R$, the same hard-edge tacnode process is obtained as in Theorem~\ref{thm:asymptotics}.

The fact that the two different ways of taking the limit gives the same result is not obvious, but we do not prove it here.
The reason why the first way is more interesting is that the Airy$_2$ process conditioned to stay below a parabola for a fixed finite interval
i.e.\ the object which arises in the intermediate step in the second approach is know, its distribution is given in~\cite{CQR11}.
\end{remark}

\section{Relation to the six-vertex model and the Aztec diamond}\label{s:6vertex}
The six-vertex model is a statistical mechanics model with short range interaction which is however sensitive to the boundary conditions.
For instance, imposing the so-called domain wall boundary conditions (DWBC), it was noticed in~\cite{KZJ00} that it has a macroscopic influence on the system.
In this setting, the model has two free parameters.
When these parameters satisfy a given equation, the system becomes ``free-fermion'' and there is a (many-to-one) mapping to the Aztec diamond~\cite{ZJ00}.
For the free-fermion case, one can associate a set of non-intersecting lines to the six-vertex configurations, from which the Aztec diamond configurations can be recovered~\cite{FS06}.
These are illustrated in Figure~\ref{Fig6vertex}.
\begin{figure}
\begin{center}
\includegraphics[angle=-45,height=3cm]{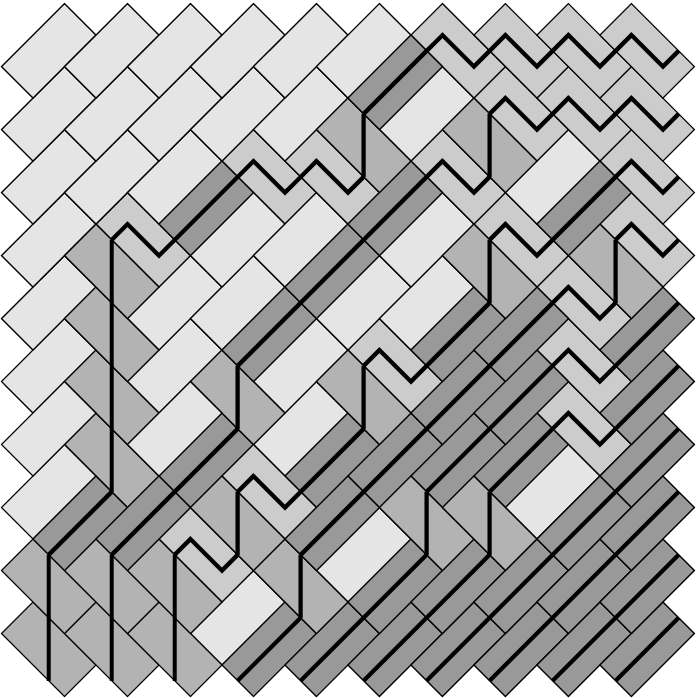}
\hspace*{2em}
\includegraphics[angle=-45,height=4cm]{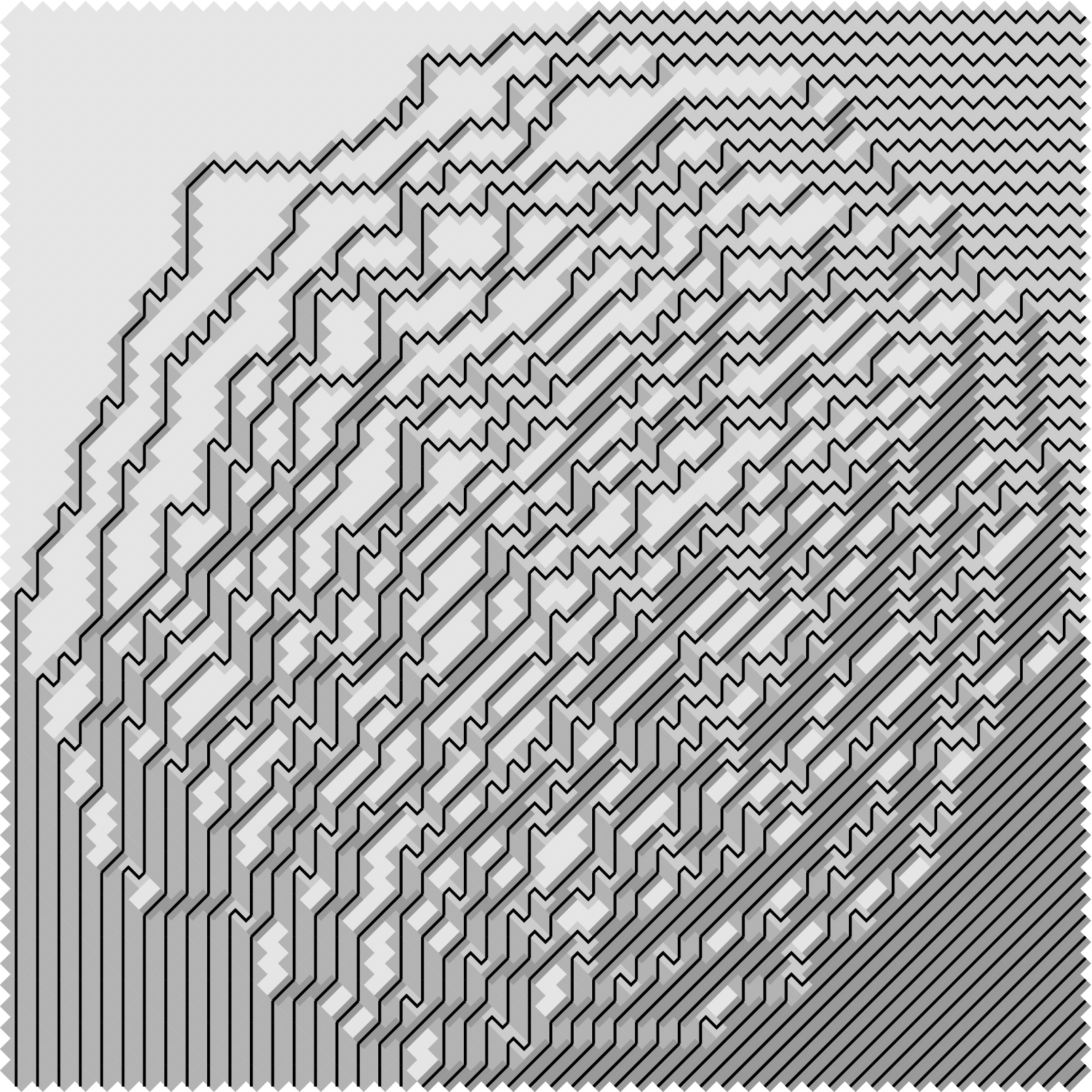}
\caption{Illustration of the non-intersecting line ensemble for an Aztec diamond/six-vertex model with DWBC of size $N=10$ (left) and size $N=50$ (right).}
\label{Fig6vertex}
\end{center}
\end{figure}

In the recent papers on the six-vertex model~\cite{CP13b,CP15,CPS16,CPS16b},
questions concerning the limit shape and correlation functions have been addressed for the six-vertex model also for other domains.
In particular, domains obtained from a square by cutting off a triangle or a rectangle from the corner were considered with DWBC.
In terms of the Aztec diamond, this corresponds to conditioning the dominoes in the top corner to be all fixed and horizontal.
The fixed dominoes form the region which has been cut out.

The Aztec diamond has been studied very well.
In particular, denote the size of the Aztec diamond by $N$.
One can think of lines in discrete time $t\in [-N,N]$.
\begin{theorem}[Theorem~1.1 of~\cite{Jo03}]
Denote by $X_N(t)$ the top line of the Aztec diamond at time $t$.
Then
\begin{equation}\label{eqAztec}
\frac{X_N(2^{-1/6}N^{2/3}T)-N/\sqrt{2}}{2^{-5/6}N^{1/3}}\to {\mathcal A}_2(T)-T^2
\end{equation}
in the sense of finite dimensional distributions.
Here ${\mathcal A}_2$ is the Airy$_2$ process.
\end{theorem}
The result is derived by analyzing the point process of the lines.
Consider the $(N^{2/3},N^{1/3})$ windows around the top line of Figure~\ref{Fig6vertex}, i.e.\ if $(t,x)$ denotes the coordinates of the lines in Figure~\ref{Fig6vertex}, one considers
\begin{equation}\label{eqScalingAztec}
(t,x)=(2^{-1/6}N^{2/3}T, N/\sqrt{2}+2^{-5/6}N^{1/3}U).
\end{equation}
Then under this scaling, the lines converge to a determinantal point process with correlation kernel given by the extended Airy kernel, see~\eqref{eqAiryKernel}.
This is the same limit as the appropriate scaling limit obtained from $N$ non-intersecting Brownian bridges as $N\to\infty$.
Notice that the scaling of the horizontal and vertical directions is compatible with the Brownian scaling
(as it is the case for the limit process since the Airy$_2$ process is locally Brownian~\cite{Ha07,CH11,CP15b}).

\emph{$L$-shaped case}: Under the scaling \eqref{eqScalingAztec},
cutting out a square from the top of the Aztec diamond such that its lower tip is at height $N/\sqrt{2}+2^{-5/6}N^{1/3}R$
is asymptotically equivalent to forbidding only a vertical line segment down to the tip of the square.
Denote by $X_N^R$ the top line in this case.
Then, from the above discussion, we expect the following:
\begin{conjecture}
Define
\begin{equation}
X_N^{R,{\rm resc}}(T)=\frac{X_N^R(2^{-1/6}N^{2/3}T)-N/\sqrt{2}}{2^{-5/6}N^{1/3}}.
\end{equation}
Then for any given $T_1<T_2<\ldots<T_k$ and $U_1,\ldots,U_k\leq R$,
\begin{equation}
\lim_{N\to\infty} \P\left(\bigcap_{\ell=1}^k \{X_N^{R,{\rm resc}}(T_\ell)\leq U_\ell\}\right)
=\frac{\P\left(\bigcap_{\ell=1}^k\{{\mathcal A}_2(T_\ell)-T_\ell^2\leq U_\ell\}\cap \{{\mathcal A}_2(0)\leq R)\}\right)}{\P({\mathcal A}_2(0)\leq R)}
\end{equation}
where ${\mathcal A}_2$ is the Airy$_2$ process~\cite{PS02,Jo03}.
As a consequence
\begin{equation}
\lim_{N\to\infty} \P(X_N^{R,{\rm resc}}\leq U)=\frac{F_{\rm GUE}(\min\{U,R\})}{F_{\rm GUE}(R)},
\end{equation}
where $F_{\rm GUE}$ is the GUE Tracy--Widom distribution function~\cite{TW94}.
\end{conjecture}

\emph{Pentagonal case}: Under the scaling \eqref{eqScalingAztec}, cutting out a triangle on the top corner at height $N/\sqrt{2}+2^{-5/6}N^{1/3}R$
becomes asymptotically a conditioning to stay below a fixed height $R$.
Denote by $X_N^R$ be the top line in this case.
Then, we expect to have the following:
\begin{conjecture}
Define
\begin{equation}
X_N^{R,{\rm resc}}(T)=\frac{X_N^R(2^{-1/6}N^{2/3}T)-N/\sqrt{2}}{2^{-5/6}N^{1/3}}.
\end{equation}
Then
\begin{equation}
\lim_{N\to\infty} \P\left(\bigcap_{\ell=1}^k \{X_N^{R,{\rm resc}}(T_\ell)\leq U_\ell\}\right) = \det\left(\id-\wh K^\ext\right)_{L^2(E)}
\end{equation}
with the set $E=\{(T_1,[U_1-R,0])\times \ldots\times (T_k,[U_k-R,0])\}$.
\end{conjecture}

\section{Multipoint distribution and heuristics for the correlation kernel}\label{s:mp}

In this section, we consider the process of $N$ non-intersecting Brownian bridges conditioned to stay below a constant level.
First we prove Theorem~\ref{thm:DistrH}, that is, the probability that this conditional process stays below a function of the form \eqref{defhr}
can be written as a Fredholm determinant of the kernel $K_\ext$.
As a consequence, we show that the multipoint distribution of the conditional process also has a Fredholm determinantal form,
which is part of the statement of Theorem~\ref{thm:corrkernel}.
This does not imply that $K_\ext$ is the correlation kernel for the point process of the non-intersecting Brownian bridges, but it gives a potential candidate for it.
The proof that $K_\ext$ is actually the correlation kernel is performed directly in Section~\ref{s:directcorr}.

Our heuristic derivation of the correlation kernel for $N$ non-intersecting Brownian bridges conditioned to stay below a constant
is based on the formula given in Theorem~\ref{thm:NBBbelowh} for the probability that the top path of $N$ Brownian bridges is below a function.
First we verify that the kernel which appears in Theorem~\ref{thm:NBBbelowh} is trace class.

\begin{lemma}\label{lemma:traceclass}
For any function $h:[0,1]\to\R$ and for any fixed integer $N$, the operator with kernel $K_N^h$ given in \eqref{defKN} is trace class on $L^2(\{0,1,\dots,N-1\})$.
\end{lemma}

\begin{proof}
Since $N$ is fixed, it is enough to show that $K_N^h(n,m)$ is finite for any $n,m<N$, that is, the double integral in \eqref{defKN} is finite.
By \eqref{defPsi}, one clearly has $|\Psi_{\tau_2}^m(v)|\le Ce^{c|v|}$ for some finite constants $C$ and $c$.
By definition \eqref{defT1}, $|T_{\tau_1,\tau_2}(u,v)|\le\phi_{2(\tau_2-\tau_1)}(y-x)$, hence
\begin{equation}
\int_\R\d v\,|T_{\tau_1,\tau_2}^h(u,v)\Psi_{\tau_2}^m(v)|\le Ce^{c|u|}
\end{equation}
for some finite constants $C$ and $c$.
On the other hand, \eqref{phirepr} shows that $\Phi_{\tau_1}^n(u)$ has a Gaussian decay in $u$, i.e.
\begin{equation}
|\Phi_{\tau_1}^n(u)|\le Ce^{-\frac{u^2}{8\tau_1}}
\end{equation}
for some finite $C$.
This completes the proof.
\end{proof}

For any function $f\in L^2(\R)$, let
\begin{equation}\label{defP}
P_\eta f(x)=\id_{x\ge\eta}f(x),\qquad\ol P_\eta f(x)=\id_{x<\eta}f(x)
\end{equation}
be the projection operators.

\begin{proof}[Proof of Theorem~\ref{thm:DistrH}]
The strategy of the proof is to compare the kernel $K_N^h$ for a general $h$ of the form \eqref{defhmp} to the one which corresponds to the constant $h\equiv r$.
For $h\equiv r$, in the second term on the right-hand side of \eqref{defKN} one has to insert
\begin{equation}
T_{\tau_1,\tau_2}^r=\ol P_0T_{\tau_1,\tau_2}\ol P_0
\end{equation}
by comparing \eqref{defT1}, \eqref{defT2} and \eqref{deftauhtilde}.
Hence the kernel $K_N^r$ for $h\equiv r$ simplifies to
\begin{equation}\label{KNr}
K_N^r(n,m)=\id(n,m)-\int_{\R_-}\d u\int_{\R_-}\d v\,\Phi_{\tau_1}^n(u)T_{\tau_1,\tau_2}(u,v)\Psi_{\tau_2}^m(v)=K_0(n,m)
\end{equation}
as a consequence of Proposition~\ref{prop:compatibility}.

Hence we can write the kernel $K_N^h$ for a general $h$ of the form \eqref{defhr} as
\begin{equation}\label{KNh-K0}
K_N^h(n,m)=K_0(n,m)+\int_\R\d u\int_\R\d v\,\Phi_{\tau_1}^n(u)\left(\ol P_0T_{\tau_1,\tau_2}\ol P_0-T_{\tau_1,\tau_2}^h\right)(u,v)\Psi_{\tau_2}^m(v).
\end{equation}

The conditional probability on the left-hand side of \eqref{NBBbelowhCond} is written as a ratio of two Fredholm determinants: using \eqref{KNr} we get
\begin{equation}\label{condprobFredholmratio}\begin{aligned}
&\P\left(B_N(t)<h(t)\mbox{ for }t\in[0,1]\bigm|B_N(t)<r\mbox{ for }t\in[0,1]\right)\\
&\qquad=\frac{\det(\id-K_N^h)_{\ell^2(\{0,1,\dots,N-1\})}}{\det(\id-K_0)_{\ell^2(\{0,1,\dots,N-1\})}}\\
&\qquad=\det\left(\id-(K_N^h-K_0)(\id-K_0)^{-1}\right)_{\ell^2(\{0,1,\dots,N-1\})}
\end{aligned}\end{equation}
where we used the multiplicative property of the determinant in the second equality.
By the cyclic property of the determinant and by using \eqref{KNh-K0}, \eqref{defKtau} and \eqref{TPsi}, one obtains
\begin{equation}\label{Fredholmratio2}\begin{aligned}
&\det\left(\id-(K_N^h-K_0)(\id-K_0)^{-1}\right)_{\ell^2(\{0,1,\dots,N-1\})}\\
&\qquad=\det\left(\id-\left(\ol P_0T_{\tau_1,\tau_2}\ol P_0-T_{\tau_1,\tau_2}^h\right)
\Psi_{\tau_2}(\id-K_0)^{-1}\Phi_{\tau_1}\right)_{L^2(\R_-)}\\
&\qquad=\det\left(\id-K_{\tau_1}+T_{\tau_1,\tau_2}^h K_{\tau_2,\tau_1}\right)_{L^2(\R_-)}
\end{aligned}\end{equation}
which completes the proof.
\end{proof}

To obtain the multipoint distribution of $N$ non-intersecting Brownian bridges conditioned to be under the constant level $r$ in the time interval $[0,1]$,
we specialize the probability that the top path of $N$ non-intersecting Brownian bridges stays below a function $h$ given by \eqref{defhr}.
Namely, for $0<t_1<\dots<t_k<1$, we consider the function
\begin{equation}\label{defhmp}
h(x)=\left\{\begin{array}{ll} r &\mbox{if } x\neq t_i\mbox{ for }i=1,\dots,k\\ h_i &\mbox{if } x=t_i\end{array}\right.
\end{equation}
for some $h_i\leq r$ for $i=1,\dots,k$.
Since $h$ given by \eqref{defhmp} is not in $H^1([0,1])$, one has to verify that Theorem~\ref{thm:NBBbelowh} can be used.
We prove the following lemma in Section~\ref{s:lemmas}.

\begin{lemma}\label{lemma:NRextend}
Theorem~\ref{thm:NBBbelowh} remains valid for a function $h$ defined in~\eqref{defhmp}.
\end{lemma}

The multipoint distribution of $N$ non-intersecting Brownian bridges conditioned to stay below a constant level can be expressed as follows.

\begin{proposition}\label{prop:pathintkernel}
Let $h$ be a function given by \eqref{defhmp}.
Then the following conditional probability for the top path of $N$ non-intersecting Brownian bridges can be written in a Fredholm determinant form as
\begin{multline}\label{mpdistr}
\P\left(B_N(t)<h(t)\mbox{ for }t\in[0,1]\bigm|B_N(t)<r\mbox{ for }t\in[0,1]\right)\\
=\det\left(\id-K_{\tau_1}+\ol P_{\eta_1}T_{\tau_1,\tau_2}\ol P_{\eta_2}\dots T_{\tau_{k-1},\tau_k}\ol P_{\eta_k}(T_{\tau_1,\tau_k})^{-1}K_{\tau_1}\right)_{L^2(\R_-)}.
\end{multline}
\end{proposition}

\begin{proof}
By Lemma~\ref{lemma:NRextend}, Theorem~\ref{thm:NBBbelowh} holds for this choice of $h$ as well.
The same steps used in the proof of Theorem~\ref{thm:DistrH} lead to the result.
We just need to replace $T^h_{\tau_1,\tau_2}$ with the corresponding expression for a general $h$ of the form \eqref{defhmp}, namely with
\begin{equation}
T_{\tau_1,\tau_k}^h=\ol P_{\eta_1}T_{\tau_1,\tau_2}\ol P_{\eta_2}T_{\tau_2,\tau_3}\dots T_{\tau_{k-1},\tau_k}\ol P_{\eta_k}
\end{equation}
from \eqref{defT1} and using \eqref{deftauhtilde} with \eqref{defQeta}.
\end{proof}

Using the result of~\cite{BCR13}, the Fredholm determinant with the path integral kernel on the right-hand side of \eqref{mpdistr} can be rewritten as in Proposition~\ref{prop:useBCR15} below.
Hence the second part of Theorem~\ref{thm:corrkernel} about the gap probabilities follows from Proposition~\ref{prop:pathintkernel} and \ref{prop:useBCR15}.
This is weaker than proving that $K_\ext$ is the correlation kernel for $N$ non-intersecting Brownian bridges conditioned to stay below level $r$.
We prove in Section~\ref{s:directcorr} that $K_\ext$ is actually the correlation kernel.

\begin{proposition}\label{prop:useBCR15}
For the Fredholm determinant on the right-hand side of \eqref{mpdistr}, the following identity hold
\begin{multline}\label{pathint=ext}
\det\left(\id-K_{\tau_1}+\ol P_{\eta_1}T_{\tau_1,\tau_2}\ol P_{\eta_2}\dots T_{\tau_{k-1},\tau_k}\ol P_{\eta_k}(T_{\tau_1,\tau_k})^{-1}K_{\tau_1}\right)_{L^2(\R_-)}\\
=\det(\id-QK^\ext)_{L^2(\{\tau_1,\dots,\tau_k\}\times\R_-)}
\end{multline}
where the extended kernel $K^\ext$ is given by \eqref{defKexttauu} and $Q$ is defined in \eqref{defQeta}.
\end{proposition}

\begin{proof}
Applying formally Theorem~3.3 of~\cite{BCR13} with $\mathcal W_{\tau_i,\tau_j}=T_{\tau_i,\tau_j}$ and with $K_{\tau_i}$ defined by \eqref{defKtau} would give \eqref{pathint=ext}.
This is however not correct because the operator $K_\tau$ with kernel given in \eqref{defKtau} is not a bounded operator and the assumptions of Theorem~3.3 of~\cite{BCR13} are not satisfied.

Hence we introduce the following conjugation in order to circumvent this issue.
Let
\begin{equation}\label{conjop}\begin{aligned}
\ol\Phi_\tau^n(u)&=e^{\frac{u^2}{C\tau}}\Phi_\tau^n(u),\\
\ol T_{\tau_i,\tau_j}(u,v)&=e^{-\frac{u^2}{C\tau_i}+\frac{v^2}{C\tau_j}}T_{\tau_i,\tau_j}(u,v),\\
\ol\Psi_\tau^m(v)&=e^{-\frac{v^2}{C\tau}}\Psi_\tau^m(v)
\end{aligned}\end{equation}
where $C$ is a sufficiently large constant which depends on $\tau_1,\dots,\tau_k$ in such a way the operators $\ol T_{\tau_i,\tau_{i+1}}$ are bounded.
The condition of boundedness of $\ol T_{\tau_i,\tau_{i+1}}$ is \mbox{$C>4(\tau_{i+1}-\tau_i)/\tau_{i+1}$},
because then the $v^2$ term in the exponent has negative sign in $\ol T_{\tau_i,\tau_{i+1}}(u,v)$ in \eqref{conjop}.
Further in this case,
\begin{equation}
\int_\R\d v\,\ol T_{\tau_i,\tau_{i+1}}(u,v)=\sqrt{\frac{C\tau_{i+1}}{C\tau_{i+1}-4(\tau_{i+1}-\tau_i)}}e^{-\frac{4(\tau_{i+1}-\tau_i)u^2}{C\tau_{i+1}(C\tau_{i+1}-4(\tau_{i+1}-\tau_i))}}
\end{equation}
which has Gaussian decay in $u$.

Replacing $\Phi_\tau^n$ and $\Psi_\tau^m$ by $\ol\Phi_\tau^n$ and $\ol\Psi_\tau^m$ in the definition \eqref{defKtau} of $K_\tau$, we get the kernel
\begin{equation}
\ol K_\tau(u,v)=e^{-\frac{u^2}{C\tau}+\frac{v^2}{C\tau}}K_\tau(u,v).
\end{equation}
Note that the Fredholm determinant on the left-hand side of \eqref{pathint=ext} does not change if the operators $K$ and $T$ are replaced by $\ol K$ and $\ol T$ since it is just a conjugation, i.e.
\begin{multline}
\det\left(\id-K_{\tau_1}+\ol P_{\eta_1}T_{\tau_1,\tau_2}\ol P_{\eta_2}\dots T_{\tau_{k-1},\tau_k}\ol P_{\eta_k}(T_{\tau_1,\tau_k})^{-1}K_{\tau_1}\right)_{L^2(\R_-)}\\
=\det\left(\id-\ol K_{\tau_1}+\ol P_{\eta_1}\ol T_{\tau_1,\tau_2}\ol P_{\eta_2}\dots\ol T_{\tau_{k-1},\tau_k}\ol P_{\eta_k}(\ol T_{\tau_1,\tau_k})^{-1}\ol K_{\tau_1}\right)_{L^2(\R_-)}.
\end{multline}

To apply Theorem~3.3 of~\cite{BCR13} (with the minor modification that now the space is $L^2(\R_-)$) with $\mathcal W_{\tau_i,\tau_j}=\ol T_{\tau_i,\tau_j}$ and with $\ol K_{\tau_i}$,
we check the three assumptions of the theorem.
For Assumption 1, all the operators which appear are bounded.
In particular, the boundedness of $\ol T_{\tau_i,\tau_j}$ was checked above.
The operator $\ol K_\tau$ is also bounded if $C>4$ by comparing the Gaussian decay $\Phi_\tau^n(u)\sim e^{-\frac{u^2}{4\tau}}$ with the conjugation \eqref{conjop}.

Assumption 2 about compatibility is rather clear using the interpretation of $T_{\tau_i,\tau_j}$ as a Brownian bridge transition kernel and by Proposition~\ref{prop:compatibility}.
Since the kernels of all the conjugated operators $\mathcal W_{\tau_i,\tau_j}$ and $\ol K_{\tau_i}$ which appear have Gaussian decay,
the trace class properties needed for Assumption 3 are straightforward to check.
Hence Theorem~3.3 of~\cite{BCR13} can be used which gives \eqref{pathint=ext} with $K_\ext$ replaced by its conjugated version on the right-hand side,
but the conjugation can be removed without changing the Fredholm determinant.
\end{proof}

\section{Extension of the Nguyen--Remenik formula}\label{s:extension}

In this section, we extend Proposition~\ref{prop:NguyenRemenik},
the Nguyen--Remenik formula for the probability that $N$ non-intersecting Brownian bridges stay below a given function on $[a,b]$ for any fixed $0<a<b<1$
to the probability that the Brownian bridges stay below a function $h$ of the form \eqref{defhr} on $[0,1]$.

\begin{proof}[Proof of Theorem~\ref{thm:NBBbelowh}]
First we express the Brownian bridge probability on the right-hand side of \eqref{defTheta} for the special choice of the function $h$ given in \eqref{defhr} in terms of $T_{\tau_1,\tau_2}^h$.
By introducing the drifted and shifted Brownian bridge $\wt b(\tau)=\wh b(\tau)-(1+4\tau)r/\sqrt2$, one can write
\begin{multline}\label{multiptprob}
\P_{\wh b(\alpha)=e^Ax,\wh b(\beta)=e^By}\left(\wh b(\tau)\le\frac{1+4\tau}{\sqrt2}h\left(\frac{4\tau}{1+4\tau}\right)\mbox{ for }\tau\in[\alpha,\beta]\right)\\
=\P_{\substack{\wt b(\alpha)=e^Ax-(1+4\alpha)r/\sqrt2\\\wt b(\beta)=e^By-(1+4\beta)r/\sqrt2}}
\left(\wt b(\tau)\le0\mbox{ for }\tau\in[\alpha,\beta],\wt b(\tau)\le\wt h(\tau)\mbox{ for }\tau\in(\tau_1,\tau_2)\right)
\end{multline}
where $\tau_1,\tau_2$ and $\wt h$ are defined by \eqref{deftauhtilde}.
Using the notation \eqref{defT2}, we can condition on the values of the Brownian bridge $\wt b(\tau)$ at times $\tau_1$ and $\tau_2$ and rewrite the right-hand side of \eqref{multiptprob} as
\begin{multline}\label{condbelow0}
\P_{\substack{\wt b(\alpha)=e^Ax-(1+4\alpha)r/\sqrt2\\\wt b(\beta)=e^By-(1+4\beta)r/\sqrt2}}
\left(\wt b(\tau)\le0\mbox{ for }\tau\in[\alpha,\beta],\wt b(\tau)\le\wt h(\tau)\mbox{ for }\tau\in(\tau_1,\tau_2)\right)\\
=\id_{\substack{x\le\sqrt2 r\cosh A\\y\le\sqrt2 r\cosh B}}
\frac{\int_{-\infty}^0\d u\int_{-\infty}^0\d v\,T_{\alpha,\tau_1}(e^Ax-\frac{(1+4\alpha)r}{\sqrt2},u)
T_{\tau_1,\tau_2}^h(u,v)T_{\tau_2,\beta}(v,e^By-\frac{(1+4\beta)r}{\sqrt2})}
{\phi_{2(\beta-\alpha)}(e^By-e^Ax-2\sqrt2(\beta-\alpha)r)}
\end{multline}
where the indicator on the right-hand side of \eqref{condbelow0} comes from the condition
that the starting point and the endpoint of the Brownian bridge $\wt b(\tau)$ should be below $0$ to get a non-zero probability.

Next we compare the operator $\Theta_{A,B}$ to the case of $N$ non-intersecting Brownian bridges not conditioned to stay below any function, that is, the free case.
We express $\Theta_{A,B}$ as the operator for the free case minus a remainder.
From the representation of $\Theta_{A,B}$ as the solution operator of a boundary value problem given in~\cite{NR15}, one obtains that
\begin{equation}\label{freeev}
e^{-(B-A)D}(x,y)=e^{(y^2-x^2)/2+B}\frac{\exp\left(-\frac{(e^By-e^Ax)^2}{4(\beta-\alpha)}\right)}{\sqrt{4\pi(\beta-\alpha)}}
\end{equation}
which corresponds to $\Theta_{A,B}$ with the choice $h=\infty$.
By defining
\begin{multline}\label{defR}
R_{A,B}(x,y)=e^{(y^2-x^2)/2+B}\frac{\exp\left(-\frac{(e^By-e^Ax)^2}{4(\beta-\alpha)}\right)}{\sqrt{4\pi(\beta-\alpha)}}\\
\times\left(1-\frac{\int_{-\infty}^0\d u\int_{-\infty}^0\d v\,T_{\alpha,\tau_1}(e^Ax-\frac{(1+4\alpha)r}{\sqrt2},u)
T_{\tau_1,\tau_2}^h(u,v)T_{\tau_2,\beta}(v,e^By-\frac{(1+4\beta)r}{\sqrt2})}
{\phi_{2(\beta-\alpha)}(e^By-e^Ax-2\sqrt2(\beta-\alpha)r)}\right),
\end{multline}
we can write the operator identity
\begin{equation}\label{Thetaproj}
\Theta_{A,B}=\ol P_{\sqrt2 r\cosh A}(e^{-(B-A)D}-R_{A,B})\ol P_{\sqrt2 r\cosh B}
\end{equation}
using the notation \eqref{defP}.

For the proof of Theorem~\ref{thm:NBBbelowh}, we need to take the limit $a\to0$ and $b\to1$ in \eqref{condbelowh}. Thus we set
\begin{equation}\label{ABL}
A=-L,\quad B=L\quad\mbox{which means}\quad a=\frac1{1+e^{2L}},\quad b=\frac{e^{2L}}{1+e^{2L}}.
\end{equation}
We decompose the operator $\Theta_{-L,L}$ as a sum of the operator which corresponds to the free case, the remainder operator and an error term as
\begin{equation}\label{Thetadecomp}
\Theta_{-L,L}=e^{-2LD}-R_{-L,L}-\Omega_L
\end{equation}
where the error term is
\begin{equation}
\Omega_L=e^{-2LD}-R_{-L,L}-\ol P_{\sqrt2 r\cosh L}(e^{-2LD}-R_{-L,L})\ol P_{\sqrt2 r\cosh L}.
\end{equation}
Since $K_{\Herm,N}$ defined by (\ref{defKHerm}) is a projector on a subspace of eigenvectors of $D$, it commutes with $e^{LD}$ and thus one has $e^{2LD}K_{\Herm,N}=(e^{LD}K_{\Herm,N})^2$.
Using the identity $\det(\id+AB)=\det(\id+BA)$, Proposition~\ref{prop:NguyenRemenik} can be written as
\begin{multline}\label{prob01}
\P(B_N(t)<h(t)\mbox{ for }t\in[0,1])\\
=\lim_{L\to\infty}\det(\id-K_{\Herm,N}+e^{LD}K_{\Herm,N}\Theta_{-L,L}e^{LD}K_{\Herm,N})_{L^2(\R)}.
\end{multline}
Next we use the decomposition \eqref{Thetadecomp} of $\Theta_{-L,L}$.
We prove the following lemma in Section~\ref{s:lemmas}.
\begin{lemma}\label{lemma:error}
The error term $\wt\Omega_L=e^{LD}K_{\Herm,N}\Omega_Le^{LD}K_{\Herm,N}$ goes to $0$ in trace norm as $L\to\infty$.
\end{lemma}
Thus, by Lemma~\ref{lemma:error}, in the $L\to\infty$ limit, we can neglect the error term in the Fredholm determinant on the right-hand side of \eqref{prob01}
(use for example Lemma~4 in Chap.~XIII.17 of~\cite{RS78IV}). Consequently one obtains
\begin{equation}\label{BNbelowh01}
\P\left(B_N(t)<h(t)\mbox{ for }t\in[0,1]\right)=\lim_{L\to\infty}\det\left(\id-e^{LD}K_{\Herm,N}R_{-L,L}e^{LD}K_{\Herm,N}\right)_{L^2(\R)}.
\end{equation}

By the definition \eqref{defKHerm}, we can write $K_{\Herm,N}=\varphi\varphi^*$
where $\varphi:L^2(\{0,1,\dots,N-1\})\to L^2(\R)$ and $\varphi^*:L^2(\R)\to L^2(\{0,1,\dots,N-1\})$ are operators that are adjoints of each other defined by
\begin{equation}
\left(\varphi f\right)(x)=\sum_{n=0}^{N-1}\varphi_n(x)f(n),
\qquad\left(\varphi^*g\right)(x)=\int_\R\d x\,\varphi_n(x)g(x).
\end{equation}
By this identity and by using the cyclic property of the Fredholm determinant again, we have
\begin{equation}\label{phiRphi}\begin{aligned}
\det\left(\id-e^{LD}K_{\Herm,N}R_{-L,L}e^{LD}K_{\Herm,N}\right)_{L^2(\R)}
&=\det\left(\id-R_{-L,L}e^{2LD}K_{\Herm,N}\right)_{L^2(\R)}\\
&=\det\left(\id-\varphi^*R_{-L,L}e^{2LD}\varphi\right)_{L^2(\{0,1,\dots,N-1\})}.
\end{aligned}\end{equation}
The rest of the proof of Theorem~\ref{thm:NBBbelowh} now follows from the Proposition~\ref{prop:kerneleq} below about the equality of kernels,
since the prefactor in front of $K_N^h$ on the right-hand side of \eqref{kerneleq} is just a conjugation which can be removed without changing the value of the corresponding Fredholm determinant.
\end{proof}

\begin{proposition}\label{prop:kerneleq}
Let $h\in H^1([a,b])$ be a function which satisfies \eqref{defhr}.
Then for any $N$ and $L$, one has
\begin{equation}\label{kerneleq}
(\varphi^*R_{-L,L}e^{2LD}\varphi)_{n,m}=\sqrt{\frac{m!}{n!}\frac{2^n}{2^m}}\frac{e^{Lm}}{e^{Ln}}K_N^h(n,m)
\end{equation}
for all $n,m=0,1,\dots,N-1$.
\end{proposition}

\begin{remark}
Notice that $K_N^h$ on the right-hand side of \eqref{kerneleq} does not depend on $L$, hence up to the conjugation neither the left-hand side does, which is a priori not at all obvious.
This fact shows that the $L\to\infty$ limit of the right-hand side of \eqref{condbelowh} with \eqref{ABL} is obtained
up to conjugation by simply removing the projections from $\Theta_{-L,L}$ in \eqref{Thetaproj}.
\end{remark}

\begin{proof}[Proof of Proposition~\ref{prop:kerneleq}]
We use the following two integral representations of the harmonic oscillator functions:
\begin{align}
\varphi_n(x)&=\sqrt{\frac{2^n}{n!}}\pi^{1/4}e^{x^2/2}\frac1{\pi\I}\int_{\I\R}\d w\,e^{w^2-2wx}w^n,\label{phirepr1}\\
\varphi_n(x)&=\sqrt{\frac{n!}{2^n}}\pi^{-1/4}e^{-x^2/2}\frac1{2\pi\I}\oint_{\Gamma_0}\d z\,\frac{e^{-z^2+2zx}}{z^{n+1}}\label{phirepr2}
\end{align}
where the integration contour $\Gamma_0$ is a small circle around $0$ with counterclockwise orientation.
To compute the kernel on the left-hand side of \eqref{kerneleq}, we substitute \eqref{defT2} in the double integral in the definition \eqref{defR} of $R_{-L,L}$.
In this way, we get the terms
\begin{equation}\label{defQi}\begin{aligned}
Q_0(u,v,X,Y)&=\phi_{2\tau_1-e^{-2L}/2}(u-X)\,\phi_{2(\tau_2-\tau_1)}(v-u)\,\phi_{e^{2L}/2-2\tau_2}(Y-v),\\
Q_1(u,v,X,Y)&=-\phi_{2\tau_1-e^{-2L}/2}(u-X)\,T_{\tau_1,\tau_2}^h(u,v)\,\phi_{e^{2L}/2-2\tau_2}(Y-v),\\
Q_2(u,v,X,Y)&=\phi_{2\tau_1-e^{-2L}/2}(u-X)\,T_{\tau_1,\tau_2}^h(u,v)\,\phi_{e^{2L}/2-2\tau_2}(Y+v),\\
Q_3(u,v,X,Y)&=\phi_{2\tau_1-e^{-2L}/2}(u+X)\,T_{\tau_1,\tau_2}^h(u,v)\,\phi_{e^{2L}/2-2\tau_2}(Y-v),\\
Q_4(u,v,X,Y)&=-\phi_{2\tau_1-e^{-2L}/2}(u+X)\,T_{\tau_1,\tau_2}^h(u,v)\,\phi_{e^{2L}/2-2\tau_2}(Y+v).
\end{aligned}\end{equation}
By simplifying the exponential prefactor with the denominator on the right-hand side of \eqref{defR}, one gets
\begin{equation}
\begin{aligned}
R_{-L,L}(x,y)&=e^{(y^2-x^2)/2+L-\sqrt2r(e^Ly-e^{-L}x)+\sinh(2L)r^2}\\
&\times\int_\R\d u\int_\R\d v\sum_{j=0}^4Q_j\left(u,v,e^{-L}x-\frac{(1+e^{-2L})r}{\sqrt2},e^Ly-\frac{(1+e^{2L})r}{\sqrt2}\right).
\end{aligned}
\end{equation}
Note that one has changed the domain of integration for $u$ and $v$ to $\R$ because of the term which corresponds to $Q_0$.
In the terms which correspond to $Q_1$--$Q_4$, $T_{\tau_1,\tau_2}^h(u,v)$ is $0$ if $u$ or $v$ is positive by \eqref{defT1}.
With these notations, the kernel on the left-hand side of \eqref{kerneleq} using both representations \eqref{phirepr1}--\eqref{phirepr2} of the harmonic oscillator functions $\varphi_n$ is equal to
\begin{equation}\label{6foldint}\begin{aligned}
&(\varphi^*R_{-L,L}e^{2LD}\varphi)_{n,m}\\
&\quad=\int_\R\d x\int_\R\d y\,\varphi_n(x)R_{-L,L}(x,y)e^{2Lm}\varphi_m(y)\\
&\quad=\sqrt{\frac{m!}{n!}\frac{2^n}{2^m}}\frac2{(2\pi\I)^2}\int_\R\d x\int_\R\d y\int_{\I\R}\d w\oint_{\Gamma_0}\d z\int_\R\d u\int_\R\d v\,e^{w^2-2wx}w^ne^{-\sqrt2r(e^Ly-e^{-L}x)}\\
&\qquad\times e^{\sinh(2L)r^2+L}\sum_{j=0}^4Q_j\left(u,v,e^{-L}x-\frac{(1+e^{-2L})r}{\sqrt2},e^Ly-\frac{(1+e^{2L})r}{\sqrt2}\right)e^{2Lm}\frac{e^{-z^2+2zy}}{z^{m+1}}.
\end{aligned}\end{equation}
Doing the change of variables
\begin{equation}
X=e^{-L}x-\frac{(1+e^{-2L})r}{\sqrt2},\quad Y=e^Ly-\frac{(1+e^{2L})r}{\sqrt2},\quad W=e^Lw,\quad Z=e^{-L}z,
\end{equation}
one obtains
\begin{equation}\label{6foldintB}\begin{aligned}
&(\ref{6foldint})=\sqrt{\frac{m!}{n!}\frac{2^n}{2^m}}\frac{e^{Lm}}{e^{Ln}}\frac2{(2\pi\I)^2}\int_\R\d X\int_\R\d Y\int_{\I\R}\d W\oint_{\Gamma_0}\d Z\int_\R\d u\int_\R\d v\,e^{W^2e^{-2L}}\\
&\times e^{-2W(X+\frac{(1+e^{-2L})r}{\sqrt2})}W^ne^{-\sqrt2r(Y-X)-\sinh(2L)r^2}\sum_{j=0}^4Q_j(u,v,X,Y)\frac{e^{-Z^2e^{2L}+2Z(Y+\frac{(1+e^{2L})r}{\sqrt2})}}{Z^{m+1}}.
\end{aligned}\end{equation}

The integral with respect to $X$ and $Y$ in \eqref{6foldintB} can be computed, since they are Gaussian integrals.
One has
\begin{equation}\label{XYGaussint}\begin{aligned}
\int_\R\d X\,\phi_{2\tau_1-e^{-2L}/2}(u\pm X)e^{-2WX+\sqrt2rX}&=e^{(4\tau_1-e^{-2L})(\sqrt2r-2W)^2/4\mp(\sqrt2r-2W)u},\\
\int_\R\d Y\,\phi_{e^{2L}/2-2\tau_2}(v\pm Y)e^{2ZY-\sqrt2rY}&=e^{(e^{2L}-4\tau_2)(\sqrt2r-2Z)^2/4\pm(\sqrt2r-2Z)v}.
\end{aligned}\end{equation}
Then putting the definitions \eqref{defQi} into \eqref{6foldintB}, using \eqref{XYGaussint} and the notation \eqref{deffg}, one gets
\begin{equation}\label{phiRphiS}\begin{aligned}
&(\varphi^*R_{-L,L}e^{2LD}\varphi)_{n,m}\\
&\quad=\sqrt{\frac{m!}{n!}\frac{2^n}{2^m}}\frac{e^{Lm}}{e^{Ln}}\frac2{(2\pi\I)^2}\int_{\I\R}\d W\oint_{\Gamma_0}\d Z\int_\R\d u\int_\R\d v\,
\frac{W^ne^{\tau_1(\sqrt2r-2W)^2-\sqrt2rW}}{Z^{m+1}e^{\tau_2(\sqrt2r-2Z)^2-\sqrt2rZ}}\\
&\qquad\times\left(f_W(u)\phi_{2(\tau_2-\tau_1)}(v-u)g_Z(v)-(f_W(u)-f_W(-u))T_{\tau_1,\tau_2}^h(u,v)(g_Z(v)-g_Z(-v))\right).
\end{aligned}\end{equation}
By using \eqref{fugu} of Lemma~\ref{lemma:fgphi}, one can see that the integral of the first term on the right-hand side of \eqref{phiRphiS} up to conjugation is
\begin{multline}\label{intid}
\frac2{(2\pi\I)^2}\int_{\I\R}\d W\oint_{\Gamma_0}\d Z\int_\R\d u\int_\R\d v\,
\frac{W^ne^{\tau_1(\sqrt2r-2W)^2-\sqrt2rW}}{Z^{m+1}e^{\tau_2(\sqrt2r-2Z)^2-\sqrt2rZ}}f_W(u)\phi_{2(\tau_2-\tau_1)}(v-u)g_Z(v)\\
=\id(n,m).
\end{multline}
Comparing \eqref{phiRphiS} and \eqref{intid} with \eqref{defKN} and \eqref{defPhi}--\eqref{defPsi} completes the proof.
\end{proof}

\section{Direct derivation of the correlation kernel}\label{s:directcorr}

In this section, we prove Theorem~\ref{thm:corrkernel} where the correlation kernel of $N$ non-intersecting Brownian bridges conditioned to stay below a constant level is determined.
The direct proof of the correlation kernel follows the line of~\cite{TW07b} where the correlation kernel for non-intersecting Brownian bridges were computed without further conditioning.

Let us define the functions
\begin{align}
\wt\Phi_t^i(x)&=\frac1{2^n\sqrt\pi}\left(\frac{1-t}t\right)^{\frac{i+1}2}
\bigg(e^{-\frac{x^2}{2t}}H_i\Big(\frac x{\sqrt{2t(1-t)}}\Big)-e^{-\frac{(2r-x)^2}{2t}}H_i\Big(\frac{2r-x}{\sqrt{2t(1-t)}}\Big)\bigg),\\
\wt\Psi_t^j(x)&=\frac1{j!}\left(\frac t{1-t}\right)^\frac j2
\bigg(e^{-\frac{x^2}{2(1-t)}}H_j\Big(\frac x{\sqrt{2t(1-t)}}\Big)-e^{-\frac{(2r-x)^2}{2(1-t)}}H_j\Big(\frac{2r-x}{\sqrt{2t(1-t)}}\Big)\bigg)
\end{align}
for $t\in[0,1]$, $x<r$ and $i$ integer where $H_i$ is the $i$th Hermite polynomial.
For any $0\le t_1<t_2\le1$ and $x,y<r$, let
\begin{equation}\label{defTtilde}
\wt T_{t_1,t_2}(x,y)=\frac1{\sqrt{2\pi(t_2-t_1)}}\left(e^{-\frac{(x-y)^2}{2(t_2-t_1)}}-e^{-\frac{(2r-x-y)^2}{2(t_2-t_1)}}\right)
\end{equation}
be the free evolution kernel of a Brownian motion below level $r$.

\begin{proposition}\label{prop:jointdensity}
Let $0<t_1<\dots<t_k<1$ be times and $x_1^{(l)}<\dots<x_N^{(l)}$ be positions, $l=1,\dots,k$.
Then the joint density of $N$ non-intersecting Brownian bridges conditioned to stay below level $r$ for $[0,1]$ at times $t_i$ and positions $x_j^{(l)}$ is proportional to
\begin{equation}\label{jointdensity}
\det\left(\wt\Phi_{t_1}^{i-1}(x_j^{(1)})\right)_{i,j=1}^N
\prod_{l=1}^{k-1}\det\left(\wt T_{t_l,t_{l+1}}(x_i^{(l)},x_j^{(l+1)})\right)_{i,j=1}^N
\det\left(\wt\Psi_{t_k}^{i-1}(x_j^{(k)})\right)_{i,j=1}^N.
\end{equation}
\end{proposition}

\begin{proof}[Proof of Proposition~\ref{prop:jointdensity}]
We follow the usual strategy to get $N$ non-intersecting Brownian bridges which start and end at $0$.
We let them start and end at positions $-\varepsilon,-2\varepsilon,\dots,-N\varepsilon$, and then we will let $\varepsilon\to 0$.
By a Karlin--McGregor type formula, their joint density is given by
\begin{equation}\label{epsdensity}
\det\left(\wt T_{0,t_1}(-i\varepsilon,x_j^{(1)})\right)_{i,j=1}^N
\prod_{l=1}^{k-1}\det\left(\wt T_{t_l,t_{l+1}}(x_i^{(l)},x_j^{(l+1)})\right)_{i,j=1}^N\det\left(\wt T_{t_k,1}(x_i^{(k)},-j\varepsilon)\right)_{i,j=1}^N.
\end{equation}
The product of $k-1$ determinants in the middle in \eqref{jointdensity} and in \eqref{epsdensity} is the same.
The general $(i,j)$ entry of the first determinant in \eqref{epsdensity} is
\begin{equation}\begin{aligned}
\wt T_{0,t_1}(-i\varepsilon,x_j^{(1)})
&=\frac1{\sqrt{2\pi t_1}}\left(e^{-\frac{(x_j^{(1)}+i\varepsilon)^2}{2t_1}}-e^{-\frac{(2r-x_j^{(1)}+i\varepsilon)^2}{2t_1}}\right)\\
&=\frac{e^{-\frac{i^2\varepsilon^2}{2t_1}}}{\sqrt{2\pi t_1}}e^{-\frac{(x_j^{(1)})^2}{2t_1}}\left(1-\frac{i\varepsilon x_j^{(1)}}{t_1}+\frac12\frac{i^2\varepsilon^2(x_j^{(1)})^2}{t_1^2}\pm\dots\right)\\
&\qquad-\frac{e^{-\frac{i^2\varepsilon^2}{2t_1}}}{\sqrt{2\pi t_1}}e^{-\frac{(2r-x_j^{(1)})^2}{2t_1}}
\left(1-\frac{i\varepsilon(2r-x_j^{(1)})}{t_1}+\frac12\frac{i^2\varepsilon^2(2r-x_j^{(1)})^2}{t_1^2}\pm\dots\right)
\end{aligned}\end{equation}
where we used Taylor expansion in the last step.

By elementary row operations with the matrix in the first determinant in \eqref{epsdensity}, one obtains
\begin{multline}\label{detmonom}
\det\left(\wt T_{0,t_1}(-i\varepsilon,x_j^{(1)})\right)_{i,j=1}^N\\
=c(\varepsilon)\left[\det\bigg(e^{-\frac{(x_j^{(1)})^2}{2t_1}}(x_j^{(1)})^{i-1}-e^{-\frac{(2r-x_j^{(1)})^2}{2t_1}}(2r-x_j^{(1)})^{i-1}\bigg)_{i,j=1}^N+\O(\varepsilon)\right]
\end{multline}
where $c(\varepsilon)$ is a constant which does not depend on the $x_j^{(1)}$ variables.
(Notice that $c(\varepsilon)$ depends on $\varepsilon$ asymptotically as $\varepsilon^{N(N-1)/2}$, but it is unimportant for the proposition.)
The determinant on the right-hand side of \eqref{detmonom} is already independent of $\varepsilon$, hence it is also the factor which appears in the $\varepsilon\to0$ limit.
By further row manipulations in the determinant on the right-hand side of \eqref{detmonom},
one can turn the monomials $(x_j^{(1)})^{i-1}$ and $(2r-x_j^{(1)})^{i-1}$ into any polynomials of degree $i-1$, but with the same polynomial for both terms.
In particular, by choosing the $(i-1)$st Hermite polynomial with rescaled argument $x\mapsto H_{i-1}(x/\sqrt{2t(1-t)})$,
one gets that the determinant on the right-hand side of \eqref{detmonom} is proportional to the first factor in \eqref{jointdensity}.
The argument for the last determinant is the same, hence the proof is complete.
\end{proof}

\begin{proposition}\label{prop:phi=phi}
With the relation \eqref{deftauhtilde} between the variables $t_1,t_2$ and $\tau_1,\tau_2$ and with \eqref{deftauu} between $x_i$ and $u_i$, one has the following equality of the conjugated functions
\begin{align}
\Phi_\tau^n(u)&=e^{-\frac{r^2}2-\frac{(x-r)^2}{2(1-t)}}\wt\Phi_t^n(x),\label{phi=phi}\\
\Psi_\tau^n(u)&=e^{\frac{r^2}2+\frac{(x-r)^2}{2(1-t)}}\wt\Psi_t^n(x),\label{psi=psi}\\
T_{\tau_1,\tau_2}(u_1,u_2)&=\sqrt{2(1-t_1)(1-t_2)}e^{\frac{(x_1-r)^2}{2(1-t_1)}-\frac{(x_2-r)^2}{2(1-t_2)}}\wt T_{t_1,t_2}(x_1,x_2).\label{t=t}
\end{align}
\end{proposition}

With Proposition~\ref{prop:phi=phi}, proving Theorem~\ref{thm:corrkernel} is easy.

\begin{proof}[Proof of Theorem~\ref{thm:corrkernel}]
The correlation kernel can be directly obtained from the general formula given in~\cite{Joh10},
since the joint density of $N$ non-intersecting Brownian bridges conditioned to stay below level $r$ is given by \eqref{jointdensity} where the functions which appear in the determinants satisfy
\begin{align}
\int_{-\infty}^r\d x\,\frac{\wt\Phi_{t_1}^i(x)}{2^{1/4}\sqrt{1-t_1}}\wt T_{t_1,t_2}(x,y)&=\frac{\wt\Phi_{t_2}^i(y)}{2^{1/4}\sqrt{1-t_2}},\\
\int_{-\infty}^r\d y\,\wt T_{t_1,t_2}(x,y)\frac{\wt\Psi_{t_2}^j(y)}{2^{1/4}\sqrt{1-t_2}}&=\frac{\wt\Psi_{t_1}^j(x)}{2^{1/4}\sqrt{1-t_1}},\\
\int_{-\infty}^r\d x\,\frac{\wt\Phi_t^i(x)}{2^{1/4}\sqrt{1-t_1}}\frac{\wt\Psi_t^j(x)}{2^{1/4}\sqrt{1-t_2}}&=(\id-K_0)(i,j)
\end{align}
which is a direct consequence of Proposition~\ref{prop:compatibility} knowing the relations proved in Proposition~\ref{prop:phi=phi}.
Hence the extended kernel can be written for $x_1,x_2\le r$ as
\begin{multline}
K_\ext(t_1,x_1;t_2,x_2)=-\id_{\tau_1<\tau_2}\wt T_{t_1,t_2}(x_1,x_2)\\
+\sum_{n,m=0}^{N-1}\frac{\wt\Psi_{t_1}^n(x_1)}{2^{1/4}\sqrt{1-t_1}}(\id-K_0)^{-1}(n,m)\frac{\wt\Phi_{t_2}^m(x_2)}{2^{1/4}\sqrt{1-t_2}}.
\end{multline}
Due to Proposition~\ref{prop:phi=phi}, one can write the correlation kernel in terms of the variables $\tau_i,u_i$ according to \eqref{deftauu}
since the extra factor $1/\sqrt{2(1-t_1)(1-t_2)}$ is the volume element.
This proves that the correlation kernel in terms of the the natural variables $\tau_i,u_i$ is given by \eqref{defKexttauu}.
It is also consistent with the definition \eqref{defKexttx} of the correlation kernel, which finishes the proof.
\end{proof}

For the proof of Proposition~\ref{prop:phi=phi}, the following representations are useful.

\begin{proposition}\label{prop:Hermiterepr}
The functions $\Phi_\tau^n$ and $\Psi_\tau^m$ admit the following representations in terms of Hermite polynomials
\begin{equation}\label{phirepr}
\Phi_\tau^n(u)=\frac1{2^{2n+1}\tau^{\frac{n+1}2}\sqrt\pi}
e^{-\left(\frac{(1+4\tau)r}{2\sqrt{2\tau}}+\frac u{2\sqrt\tau}\right)^2}H_n\Big(\frac{(1+4\tau)r}{2\sqrt{2\tau}}+\frac u{2\sqrt\tau}\Big)e^{2\tau r^2+\sqrt2ru}-(u\leftrightarrow -u)
\end{equation}
and
\begin{equation}\label{psirepr}
\Psi_\tau^m(u)=\frac{(2\sqrt\tau)^m}{m!} H_m\Big(\frac{(1+4\tau)r}{2\sqrt{2\tau}}+\frac u{2\sqrt\tau}\Big)e^{-2\tau r^2-\sqrt2ru}-(u\leftrightarrow -u).
\end{equation}
The notation $(u\leftrightarrow -u)$ means that we have the same term with $u$ replaced by $-u$.
\end{proposition}

\begin{proof}[Proof of Proposition~\ref{prop:Hermiterepr}]
Expanding the exponent of \eqref{defPhi} and doing the change of variables $2\sqrt\tau W=w$, one gets
\begin{equation}\label{Hermiterepr1}\begin{aligned}
\Phi_\tau^n(u)&=\frac1{\pi\I}\int_{\I\R}\d W\,W^ne^{4\tau W^2-(1+4\tau)\sqrt2rW+2\tau r^2} e^{-2uW+\sqrt2ru}-(u\leftrightarrow -u)\\
&=\frac1{(2\sqrt\tau)^{n+1}}\frac1{\pi\I}\int_{\I\R}\d w\,w^ne^{w^2-2\left(\frac{(1+4\tau)r}{2\sqrt{2\tau}}+\frac u{2\sqrt\tau}\right)w+2\tau r^2+\sqrt2ru}-(u\leftrightarrow -u).
\end{aligned}\end{equation}
By \eqref{defphi} and \eqref{phirepr1}, one has
\begin{equation}\label{Hermitereprw}
\frac1{\pi\I}\int_{\I\R}\d w\,w^ne^{w^2-2xw}=\frac1{2^n\sqrt\pi}e^{-x^2}H_n(x).
\end{equation}
Then the integral on the right-hand side of \eqref{Hermiterepr1} can be expressed with Hermite polynomials using \eqref{Hermitereprw} which immediately yields \eqref{phirepr}.

The representation \eqref{psirepr} is proved similarly.
With the change of variables $2\sqrt\tau Z=z$ in \eqref{defPsi}, one obtains
\begin{equation}\label{Hermiterepr2}\begin{aligned}
\Psi_\tau^m(u)&=\frac1{2\pi\I}\oint_{\Gamma_0}\d Z\,Z^{-(m+1)}e^{-4\tau Z^2+(1+4\tau)\sqrt2rZ-2\tau r^2}e^{2uZ-\sqrt2ru}-(u\leftrightarrow -u)\\
&=(2\sqrt\tau)^m\frac1{2\pi\I}\oint_{\Gamma_0}\d z\,z^{-(m+1)}e^{-z^2+2\left(\frac{(1+4\tau)r}{2\sqrt{2\tau}}+\frac u{2\sqrt\tau}\right)z-2\tau r^2-\sqrt2ru}-(u\leftrightarrow -u).
\end{aligned}\end{equation}
Using \eqref{defphi} and the representation \eqref{phirepr2} yields
\begin{equation}\label{Hermitereprz}
\frac1{2\pi\I}\oint_{\Gamma_0}\d z\,\frac{e^{-z^2+2zx}}{z^{n+1}}=\frac1{n!}H_n(x)
\end{equation}
Then the two integrals on the right-hand side of \eqref{Hermiterepr2} are rewritten with \eqref{Hermitereprz} which proves \eqref{psirepr}.
\end{proof}

\begin{proof}[Proof of Proposition~\ref{prop:phi=phi}]
We proceed by direct computation.
To prove \eqref{phi=phi}, one can first rearrange the right-hand side to get
\begin{equation}\label{phicompute}\begin{aligned}
&e^{-\frac{r^2}2-\frac{(x-r)^2}{2(1-t)}}\wt\Phi_t^n(x)\\
&=C_{n,t}e^{-\frac{r^2}2-\frac{(x-r)^2}{2(1-t)}}
\left(e^{-\frac{x^2}{2t}}H_n\Big(\frac x{\sqrt{2t(1-t)}}\Big)-e^{-\frac{(2r-x)^2}{2t}}H_n\Big(\frac{2r-x}{\sqrt{2t(1-t)}}\Big)\right)\\
&=C_{n,t}\left(e^{-\frac{x^2}{2t(1-t)}}H_n\Big(\frac x{\sqrt{2t(1-t)}}\Big)e^{\frac{rx}{1-t}+\frac{(t-2)r^2}{2(1-t)}}
-e^{-\frac{(2r-x)^2}{2t(1-t)}}H_n\Big(\frac{2r-x}{\sqrt{2t(1-t)}}\Big)e^{-\frac{rx}{1-t}+\frac{(t+2)r^2}{2(1-t)}}\right)
\end{aligned}\end{equation}
where $C_{n,t}=\frac1{2^n\sqrt\pi}\left(\frac{1-t}t\right)^{\frac{n+1}2}$.

Next we rewrite the right-hand side of \eqref{phicompute} in terms of the variables $\tau$ and $u$.
Using \eqref{deftauu}, one has
\begin{equation}\label{phieq1}
\frac x{\sqrt{2t(1-t)}}=\frac{(1+4\tau)r}{2\sqrt{2\tau}}+\frac u{2\sqrt\tau},\qquad\frac{2r-x}{\sqrt{2t(1-t)}}=\frac{(1+4\tau)r}{2\sqrt{2\tau}}-\frac u{2\sqrt\tau}
\end{equation}
and
\begin{equation}\label{phieq2}
\pm\frac{rx}{1-t}+\frac{(t\mp2)r^2}{2(1-t)}=2\tau r^2\pm\sqrt2ru.
\end{equation}
By substituting \eqref{deftauhtilde}, \eqref{phieq1} and \eqref{phieq2} on the right-hand side of \eqref{phicompute}, one exactly gets the representation \eqref{phirepr}, which proves \eqref{phi=phi}.

Similarly,
\begin{equation}\label{psicompute}\begin{aligned}
&e^{\frac{r^2}2+\frac{(x-r)^2}{2(1-t)}}\wt\Psi_t^n(x)\\
&\quad=\frac1{n!}\left(\frac t{1-t}\right)^{\frac n2}
e^{\frac{r^2}2+\frac{(x-r)^2}{2(1-t)}}
\bigg(e^{-\frac{x^2}{2(1-t)}}H_n\Big(\frac x{\sqrt{2t(1-t)}}\Big)-e^{-\frac{(2r-x)^2}{2(1-t)}}H_n\Big(\frac{2r-x}{\sqrt{2t(1-t)}}\Big)\bigg)\\
&\quad=\frac1{n!}\left(\frac t{1-t}\right)^{\frac n2}
\bigg(H_n\Big(\frac x{\sqrt{2t(1-t)}}\Big)e^{-\frac{rx}{1-t}-\frac{(t-2)r^2}{2(1-t)}}-H_n\Big(\frac{2r-x}{\sqrt{2t(1-t)}}\Big)e^{\frac{rx}{1-t}-\frac{(t+2)r^2}{2(1-t)}}\bigg).
\end{aligned}\end{equation}
Then by \eqref{deftauhtilde}, \eqref{phieq1} and \eqref{phieq2}, one can write the right-hand side of \eqref{psicompute} in terms of the variables $\tau$ and $u$.
Comparing this with \eqref{psirepr}, \eqref{psi=psi} is proved.

Finally, by definition \eqref{defTtilde} and by using \eqref{deftauu},
\begin{equation}\begin{aligned}
&\sqrt{2(1-t_1)(1-t_2)}e^{\frac{(x_1-r)^2}{2(1-t_1)}-\frac{(x_2-r)^2}{2(1-t_2)}}\wt T_{t_1,t_2}(x_1,x_2)\\
&\quad=\sqrt{\frac{(1-t_1)(1-t_2)}{\pi(t_2-t_1)}}e^{(1-t_1)u_1^2-(1-t_2)u_2^2}\left(e^{-\frac{((1-t_1)u_1-(1-t_2)u_2)^2}{t_2-t_1}}-e^{-\frac{((1-t_1)u_1+(1-t_2)u_2)^2}{t_2-t_1}}\right)\\
&\quad=\sqrt{\frac{(1-t_1)(1-t_2)}{\pi(t_2-t_1)}}\left(e^{-\frac{(1-t_1)(1-t_2)}{t_2-t_1}(u_1-u_2)^2}-e^{-\frac{(1-t_1)(1-t_2)}{t_2-t_1}(u_1+u_2)^2}\right).
\end{aligned}\end{equation}
By noticing that
\begin{equation}
4(\tau_2-\tau_1)=\frac{t_2-t_1}{(1-t_1)(1-t_2)},
\end{equation}
the proof is complete.
\end{proof}

\section{Asymptotics}\label{s:asymptotics}

This section is devoted to the proof of Theorem~\ref{thm:asymptotics} and Theorem~\ref{thm:asymptoticsB}.
To this end, we start with a lemma which contains the asymptotic properties of the harmonic oscillator functions which are necessary for further proofs.

\begin{lemma}\label{lemma:phiasymp}
For the $n$th harmonic oscillator function $\varphi_n$, one has
\begin{equation}\label{phiconv}
\lim_{n\to\infty}2^{-1/4}n^{1/12}\varphi_n\left(\sqrt{2n}+\frac{sn^{-1/6}}{\sqrt2}\right)=\Ai(s)
\end{equation}
uniformly on any compact subset of $\R$ for $s$.
Further, for any $c>0$, there are $s_0$ and $n_0$ such that for any $s\ge s_0$ and $n\ge n_0$,
\begin{equation}\label{phibound}
\left|2^{-1/4}n^{1/12}\varphi_n\left(\sqrt{2n}+\frac{sn^{-1/6}}{\sqrt2}\right)\right|\le e^{-cs}.
\end{equation}
There is a universal constant $C$ such that for any $n\ge1$,
\begin{equation}\label{unifphibound}
\sup_{x\in\R}\left|2^{-1/4}n^{1/12}\varphi_n(x)\right|\le C.
\end{equation}
\end{lemma}

\begin{proof}[Proof of Lemma~\ref{lemma:phiasymp}]
The formula \eqref{phiconv} is well-known, see e.g.~\cite{Sze67}.
It can be also seen in Lemma~5.8 of~\cite{FSW15}, while \eqref{phibound} can be derived directly from Lemma~5.9 of~\cite{FSW15}.
Using Definition 5.7 of~\cite{FSW15}, one has
\begin{equation}\begin{aligned}
\alpha_n(0,s)&=n^{1/3}e^{3n/2+sn^{1/3}}\frac1{2\pi\I}\int_{\I\R}\d w\,e^{nw^2/2+(2n+sn^{1/3})w}(-w)^n\\
&=2^{n-1/2}n^{-n/2-1/6}e^{3n/2+sn^{1/3}}\frac1{\pi\I}\int_{\I\R}\d W\,e^{W^2-2(\sqrt{2n}+sn^{1/3}/\sqrt2)W}W^n\\
&=2^{n-1/2}n^{-n/2-1/6}e^{3n/2+sn^{1/3}}\sqrt\frac{n!}{2^n}\pi^{-1/4}e^{(\sqrt{2n}+sn^{1/3}/\sqrt2)^2/2}\varphi_n\left(\sqrt{2n}+\frac{sn^{-1/6}}{\sqrt2}\right)\\
&=2^{-1/4}n^{1/12}\varphi_n\left(\sqrt{2n}+\frac{sn^{-1/6}}{\sqrt2}\right)
\end{aligned}\end{equation}
with the change of variables $W=-\sqrt{n/2}w$ in the second equality, with the use of \eqref{phirepr1} in the third and by Stirling's formula in the last one.
Hence the results of~\cite{FSW15} apply and one gets \eqref{phiconv} and \eqref{phibound} with $c=1$.
By inspecting the proof of Lemma~5.9 in~\cite{FSW15}, one can realize that the terms which appear in the integral representation of $\beta_t(r,s)$ in \mbox{(5.40)--(5.42)} of~\cite{FSW15}
are bounded by a large constant times $\exp(-s^{3/2})$ which is less than $e^{-cs}$ for any $c$ if $s$ is large enough.
By the last remark after (5.47) in the proof of Proposition~5.9 in~\cite{FSW15}, one gets that the same bound applies for $\alpha_t(r,s)$ as required.

Finally, \eqref{unifphibound} is an easy consequence of the detailed bound obtained in~\cite{Kra04},
see also (A.54) of~\cite{FerPhD} where $p_k(x)=H_k(x)$ (except for a small typo: $2^{2/k}$ should be $2^{k/2}$).
By replacing $x/\sqrt{2N}$ by $x$ and by \eqref{defphi}, one exactly gets \eqref{unifphibound}.
\end{proof}

Then one has the following limits as $N\to\infty$ and bounds for the functions which appear in the kernel of $N$ non-intersecting Brownian bridges.

\begin{proposition}\label{prop:asymptotics}
Consider the scaling
\begin{equation}\label{eqscalingB}
u=\frac{U N^{-1/6}}{\sqrt{2}},\quad n=N-\xi N^{1/3},\quad m=N-\zeta N^{1/3}
\end{equation}
as well as \eqref{scaling} for $t_i,r,x_i$.
Then as $N\to\infty$, it holds
\begin{align}
\lim_{N\to\infty}\left(\frac2N\right)^{\frac n2}\frac{N^{-1/6}}{\sqrt2}e^{N+\frac{RN^{1/3}}2}\Phi_\tau^n(u)&=\wh\Phi_T^\xi(U),\label{phiasymp}\\
\lim_{N\to\infty}\left(\frac N2\right)^{\frac m2}N^{1/3}e^{-N-\frac{RN^{1/3}}2}\Psi_\tau^m(u)&=\wh\Psi_T^\zeta(U).\label{psiasymp}
\end{align}
For the rescaled kernel the convergence
\begin{equation}\label{K0conv}
\lim_{N\to\infty}\left(\frac2N\right)^{\frac{n-m}2}N^{1/3}K_0(n,m)=\wh K_0(\xi,\zeta)
\end{equation}
holds.

Furthermore, for $U,V$ in a compact interval and for any $c>0$, there is a $C=C(c)$ such that the bounds
\begin{align}
\left|\left(\frac2N\right)^{\frac n2}\frac{N^{-1/6}}{\sqrt2}e^{N+\frac{RN^{1/3}}2}\Phi_\tau^n(u)\right|&\le Ce^{-c\xi},\label{Phibound}\\
\left|\left(\frac N2\right)^{\frac m2}N^{1/3}e^{-N-\frac{RN^{1/3}}2}\Psi_\tau^m(u)\right|&\le Ce^{-c\zeta},\label{Psibound}\\
\left|\left(\frac2N\right)^{\frac{n-m}2}N^{1/3}K_0(n,m)\right|&\le Ce^{-c(\xi+\zeta)}\label{K0bound}
\end{align}
hold for $\xi$ and $\zeta$ uniformly in $[0,N^{2/3}]$.
\end{proposition}

Theorem~\ref{thm:asymptotics} is now an easy consequence of Proposition~\ref{prop:asymptotics}.

\begin{proof}[Proof of Theorem~\ref{thm:asymptotics}]
It is enough to prove \eqref{Kextconv} in terms of the variables $\tau_i,u_i$, that is,
\begin{equation}
\lim_{N\to\infty}\frac{N^{-1/6}}{\sqrt2}K^\ext(\tau_1,u_1;\tau_2,u_2)=\wh K^\ext(T_1,U_1;T_2,U_2).
\end{equation}
First of all, \eqref{scaling}, \eqref{eqscalingB} and Brownian scaling give
\begin{equation}
\frac{N^{-1/6}}{\sqrt2}T_{\tau_1,\tau_2}(u_1,u_2)=T_{T_1,T_2}(U_1,U_2).
\end{equation}
Further, by the uniform decay properties \eqref{Phibound}--\eqref{K0bound} in $\xi$ and $\zeta$, in the sum for $n$ and $m$ in \eqref{defKexttauu}, dominated convergence can be used.
We thus replace the rescaled functions $\Psi_{\tau_1}^n(u_1)$, $\Phi_{\tau_2}^m(u_2)$ and the rescaled resolvent of the kernel $K_0$ in \eqref{defKexttauu} according to \eqref{phiasymp}--\eqref{K0conv}.
The conjugations and prefactors exactly cancels.
We turn the Riemann sum into an integral and by dominated convergence, we obtain \eqref{Kextconv}.
\end{proof}

\begin{proof}[Proof of Theorem~\ref{thm:asymptoticsB}]
First note that the left-hand side of \eqref{eq2.38} before taking the $N\to\infty$ limit is equal to the left-hand side of \eqref{NBBbelowhCond} for
\begin{equation}\label{htscaling}
h(t)=\sqrt N+\frac{H(T) N^{-1/6}}2,\qquad t=\frac{1+T N^{-1/3}}2.
\end{equation}
Thus we can apply Theorem~\ref{thm:DistrH}.
Under the same scaling as in the proof of Theorem~\ref{thm:asymptotics}, in particular with $u_i=\frac{U_i N^{-1/6}}{\sqrt{2}}$, Proposition~\ref{prop:asymptotics} yields
\begin{equation}
\frac{N^{-1/6}}{\sqrt2}K_{\tau_1,\tau_2}(u_1,u_2)\to\wh K_{T_1,T_2}(U_1,U_2).
\end{equation}
In order to compute the limit of $T^h$, observe that with the scaling \eqref{htscaling} and $r=\sqrt N+\frac12RN^{-1/6}$ and by \eqref{deftauhtilde},
\begin{equation}\label{tauhscaling}
\tau=\tau(T)=\frac14\frac{1+TN^{-1/3}}{1-TN^{-1/3}},\qquad\wt h(\tau)=\frac{N^{-1/6}}{\sqrt2}(H(T)-R)(1+\O(N^{-1/3})).
\end{equation}
Further, inserting \eqref{tauhscaling} and $u_i=\frac{U_i N^{-1/6}}{\sqrt{2}}$ into $T^h$ given by \eqref{defT1}, we obtain
\begin{equation}\begin{aligned}
&\lim_{N\to\infty}\frac{N^{-1/6}}{\sqrt{2}}T^h_{\tau_1,\tau_2}(u_1,u_2)\\
&\qquad=\lim_{N\to\infty}\frac{N^{-1/6}}{\sqrt{2}}\frac{\d}{\d u_2} \P_{\wt b(\tau_1)=u_1}
\left(\wt b(\tau(T))\le \wt h(\tau(T))\mbox{ for }T\in[T_1,T_2],\wt b(\tau_2)\leq u_2\right)\\
&\qquad=\frac{\d}{\d U_2} \P_{B(T_1)=U_1}\left(B(T)\le(H(T)-R)\mbox{ for }T\in[T_1,T_2],B(T_2)\leq U_2\right)\\
&\qquad=\wh T_{T_1,T_2}^{H-R}(U_1,U_2)
\end{aligned}\end{equation}
where we used the Brownian scaling by writing the probability in terms of the Brownian motion $B(T)=\frac1{\sqrt2}N^{-1/6}\wt b(\frac14+\frac12TN^{-1/3})$ in the second equality.
The convergence of the Fredholm determinant follows from the bounds of Proposition~\ref{prop:asymptotics} in the same way as for the convergence of the Fredholm determinant of Theorem~\ref{thm:asymptotics}.
\end{proof}

\begin{proof}[Proof of Proposition~\ref{prop:asymptotics}]
In the representation of $\Phi$ in terms of Hermite polynomials \eqref{phirepr}, we use \eqref{defphi}.
Then we get
\begin{multline}\label{phiwithphi}
\Phi_\tau^n(u)=\frac{\sqrt{n!}}{2^{\frac{3n}2+1}\tau^{\frac{n+1}2}\pi^{\frac14}}
e^{-\frac12\left(\frac{(1+4\tau)r}{2\sqrt{2\tau}}+\frac u{2\sqrt\tau}\right)^2}\varphi_n\Big(\frac{(1+4\tau)r}{2\sqrt{2\tau}}+\frac u{2\sqrt\tau}\Big)e^{2\tau r^2+\sqrt2ru}-(u\leftrightarrow -u).
\end{multline}
Using the scaling of the variables \eqref{scaling}, \eqref{eqscalingB}, as well as \eqref{deftauhtilde}, one has
\begin{equation}\label{usefulasymp}\begin{aligned}
\frac{(1+4\tau)r}{2\sqrt{2\tau}}\pm\frac u{2\sqrt\tau}&=\sqrt{2N}+\frac{(T^2+R\pm U)N^{-1/6}}{\sqrt2}\mp\frac{TUN^{-1/2}}{\sqrt2}+o(N^{-1/2}),\\
2\tau r^2\pm\sqrt2ru&=\frac N2+TN^{2/3}+\left(T^2+\frac R2\pm U\right)N^{1/3}+RT+T^3+o(1).
\end{aligned}\end{equation}
Further, by the scaling \eqref{scaling}, \eqref{eqscalingB} and \eqref{deftauhtilde},
\begin{equation}\label{4tau}
(4\tau)^{\frac{n+1}2}=e^{TN^{2/3}+T^3/3-T\xi+o(1)}.
\end{equation}
Finally, Stirling's formula leads to
\begin{equation}\label{stirling}
\sqrt{n!}=N^{\frac n2}(1-\xi N^{-2/3})^{\frac{N-\xi N^{1/3}}2}e^{-\frac N2-\frac{\xi N^{1/3}}2+o(1)}(2\pi N)^{\frac14}=(2\pi)^{\frac14}N^{\frac n2+\frac14}e^{-\frac N2+o(1)}.
\end{equation}
Plugging \eqref{usefulasymp}--\eqref{stirling} into \eqref{phiwithphi}, one has
\begin{multline}\label{Phicapitals}
\Phi_\tau^n(u)=\left(\frac N2\right)^{\frac n2}\sqrt2N^{1/6}e^{-N-\frac{RN^{1/3}}2+o(1)}e^{\frac23T^3+(R+\xi+U)T}\\
\times 2^{-1/4}N^{1/12}\varphi_{N-\xi N^{1/3}}\left(\sqrt{2N}+\frac{(T^2+R+U)N^{-1/6}}{\sqrt2}\right)-(U\leftrightarrow -U).
\end{multline}
On the other hand, by \eqref{phiconv} from Lemma~\ref{lemma:phiasymp}, for any $\xi>0$ fixed,
\begin{equation}\label{phimodifiedindex}
\lim_{N\to\infty}2^{-1/4}N^{1/12}\varphi_{N-\xi N^{1/3}}\left(\sqrt{2N}+\frac{sN^{-1/6}}{\sqrt2}\right)=\Ai(s+\xi).
\end{equation}
Using the notation \eqref{defPhihat}, this proves \eqref{phiasymp}.

The proof of \eqref{psiasymp} is similar.
Using \eqref{defphi} in \eqref{psirepr} gives
\begin{multline}\label{psiwithphi}
\Psi_\tau^m(u)=\frac{2^{\frac{3m}2}\tau^\frac m2\pi^{\frac14}}{\sqrt{m!}}
e^{\frac12\left(\frac{(1+4\tau)r}{2\sqrt{2\tau}}+\frac u{2\sqrt\tau}\right)^2}\varphi_m\Big(\frac{(1+4\tau)r}{2\sqrt{2\tau}}+\frac u{2\sqrt\tau}\Big)e^{-2\tau r^2-\sqrt2ru}-(u\leftrightarrow -u).
\end{multline}
Substituting \eqref{usefulasymp}--\eqref{stirling} with $n$ replaced by $m$, one has
\begin{multline}\label{Psicapitals}
\Psi_\tau^m(u)=\left(\frac2N\right)^{\frac m2}N^{-1/3}e^{N+\frac{RN^{1/3}}2+o(1)}e^{-\frac23T^3-(R+\zeta+U)T}\\
\times 2^{-1/4}N^{1/12}\varphi_{N-\zeta N^{1/3}}\left(\sqrt{2N}+\frac{(T^2+R+U)N^{-1/6}}{\sqrt2}\right)-(U\leftrightarrow -U)
\end{multline}
which proves \eqref{psiasymp} by using \eqref{defPhihat}.

To show \eqref{K0conv}, one uses \eqref{fug-u}.
By comparing the definitions \eqref{defPhi}--\eqref{defPsi} with \eqref{Phicapitals} and \eqref{Psicapitals}, one can write the kernel as
\begin{equation}\label{K0withphis}\begin{aligned}
K_0(n,m)&=\left(\frac N2\right)^{\frac{n-m}2}2^{-1/2}N^{-1/6}e^{T(\xi-\zeta)+o(1)}\\
&\qquad\times\int_\R\d U\,e^{2TU}\varphi_{N-\xi N^{1/3}}\Big(\sqrt{2N}+\frac{(T^2+R+U)N^{-1/6}}{\sqrt2}\Big)\\
&\qquad\qquad\times\varphi_{N-\zeta N^{1/3}}\Big(\sqrt{2N}+\frac{(T^2+R-U)N^{-1/6}}{\sqrt2}\Big)
\end{aligned}\end{equation}
where we made the change of variables $u=UN^{-1/6}/\sqrt2$.
For any $c>0$, there is a uniform constant $C=C(c)$ such that
\begin{equation}\label{phixibound}
\left|2^{-1/4}N^{1/12}\varphi_{N-\xi N^{1/3}}\Big(\sqrt{2N}+\frac{sN^{-1/6}}{\sqrt2}\Big)\right|\le Ce^{-c(\xi+s)}
\end{equation}
for $s>0$, because of \eqref{phibound} with $n=N-\xi N^{1/3}$.

If $n$ does not grow to infinity with $N\to\infty$, then by definition \eqref{defphi} with \mbox{$x=\sqrt{2N}+sN^{-1/6}/\sqrt2$},
the harmonic oscillator function on the left-hand side of \eqref{phixibound} is of order $e^{-N}$ which is even smaller than the right-hand side.
Using \eqref{unifphibound}, the left-hand side of \eqref{phixibound} is at most a uniform constant for $s\le0$.
Hence we can use dominated convergence in \eqref{K0withphis} and conclude that
\begin{equation}\begin{aligned}
\left(\frac2N\right)^{\frac{n-m}2}N^{1/3}K_0(n,m)
&\to e^{T(\xi-\zeta)}\int_\R\d U\,e^{2TU}\Ai(T^2+R+\xi+U)\Ai(T^2+R+\zeta-U)\\
&=2^{-1/3}\Ai(2^{-1/3}(2R+\xi+\zeta)).
\end{aligned}\end{equation}
In the last step we use the following identity: for any $s_1,s_2,t\in\R$,
\begin{equation}
\int_\R\d\lambda\,e^{t\lambda}\Ai(s_2+\lambda)\Ai(s_1-\lambda)=2^{-1/3}e^{\frac12(s_1-s_2)t}\Ai\left(2^{-1/3}\left(s_1+s_2-\frac{t^2}2\right)\right),
\end{equation}
which follows from (A.5)--(A.6) of~\cite{BFS07} using the notation (A.1) in~\cite{BFS07}.
This proves \eqref{K0conv}.

Using the uniformity of the bound \eqref{phixibound} in $\xi$, the exponential bounds in $\xi$ and $\zeta$
which can be given for \eqref{Phicapitals}, \eqref{Psicapitals} and for \eqref{K0withphis} yield \eqref{Phibound}--\eqref{K0bound}.
This completes the proof.
\end{proof}

\section{Proof of lemmas}\label{s:lemmas}

In this section, we give the proofs of all those propositions and lemmas which were found to be technical to give immediately.
For the proof of Proposition~\ref{prop:compatibility} we will use the following lemma.
\begin{lemma}\label{lemma:fgphi}
With the notation \eqref{deffg}, for any $u,v\in\R$, one has
\begin{align}
\int_\R\d u\,f_W(u)\phi_{2(\tau_2-\tau_1)}(v-u)&=e^{(\tau_2-\tau_1)(\sqrt2r-2W)^2}f_W(v),\label{fphi}\\
\int_\R\d v\,\phi_{2(\tau_2-\tau_1)}(v-u)g_Z(v)&=e^{(\tau_2-\tau_1)(\sqrt2r-2Z)^2}g_Z(u).\label{phig}
\end{align}
Further, for any $\tau>0$ and integers $n$ and $m$,
\begin{align}
\frac2{(2\pi\I)^2}\int_{\I\R}\d W\oint_{\Gamma_0}\d Z\int_\R\d u\,\frac{W^ne^{\tau\left(\sqrt2r-2W\right)^2-\sqrt2rW}}{Z^{m+1}e^{\tau\left(\sqrt2r-2Z\right)^2-\sqrt2rZ}}f_W(u)g_Z(u)
&=\id(n,m),\label{fugu}\\
\frac2{(2\pi\I)^2}\int_{\I\R}\d W\oint_{\Gamma_0}\d Z\int_\R\d u\,\frac{W^ne^{\tau\left(\sqrt2r-2W\right)^2-\sqrt2rW}}{Z^{m+1}e^{\tau\left(\sqrt2r-2Z\right)^2-\sqrt2rZ}}f_W(u)g_Z(-u)
&=K_0(n,m).\label{fug-u}
\end{align}
\end{lemma}

\begin{proof}[Proof of Proposition~\ref{prop:compatibility}]
We substitute the definition \eqref{defT2} of $T_{\tau_1,\tau_2}$ and by combining terms after the change of variables $u\to-u$, one gets
\begin{equation}\begin{aligned}
&\int_{\R_-}\d u\,(f_W(u)-f_W(-u))T_{\tau_1,\tau_2}(u,v)\\
&\qquad=\int_\R\d u\,f_W(u)\phi_{2(\tau_2-\tau_1)}(v-u)-\int_\R\d u\,f_W(u)\phi_{2(\tau_2-\tau_1)}(v+u)\\
&\qquad=e^{(\tau_2-\tau_1)(\sqrt2r-2W)^2}(f_W(v)-f_W(-v))
\end{aligned}\end{equation}
where \eqref{fphi} was used in the second equality.
This proves \eqref{PhiT}.
The proof of \eqref{TPsi} is similar.
The identity \eqref{PhiPsi} immediately follows from \eqref{fugu}--\eqref{fug-u} after the combination of the terms which appear in \eqref{defPhi}--\eqref{defPsi} and by the change of variables $u\to-u$.
\end{proof}

\begin{proof}[Proof of Lemma~\ref{lemma:fgphi}]
The identities \eqref{fphi} and \eqref{phig} are Gaussian integrals which are straightforward to compute.

To show \eqref{fugu}, one separates the integral with respect to $u$ restricted to $\R_-$ and to $\R_+$.
We can suppose that $\Gamma_0$ is so small that $\Re(Z)\in(-1,1)$ along $Z\in\Gamma_0$.
Then in the integral on $\R_-$, one can deform the $W$ contour to $-1+\I\R$, and with this, the integral with respect to $u$ can be computed as
\begin{equation}
\int_{\R_-}\d u\,f_W(u)g_Z(u)=\int_{\R_-}\d u\,e^{2(Z-W)u}=\frac1{2(Z-W)}
\end{equation}
since $\Re(Z-W)>0$ for any $Z\in\Gamma_0$ and $W\in-1+\I\R$.
Similarly on $\R_+$, one deforms the $W$ contour to $1+\I\R$, and then
\begin{equation}
\int_{\R_+}\d u\,f_W(u)g_Z(u)=\int_{\R_+}\d u\,e^{2(Z-W)u}=-\frac1{2(Z-W)}
\end{equation}
since $\Re(Z-W)<0$ in this case.
By joining the two integration contours for $W$ and by Cauchy's theorem, one gets
\begin{equation}\begin{aligned}
&\frac2{(2\pi\I)^2}\int_\R\d u\int_{\I\R}\d W\oint_{\Gamma_0}\d Z\,\frac{W^ne^{\tau\left(\sqrt2r-2W\right)^2-\sqrt2rW}}{Z^{m+1}e^{\tau\left(\sqrt2r-2Z\right)^2-\sqrt2rZ}}f_W(u)g_Z(u)\\
&\qquad=\frac1{2\pi\I}\oint_{\Gamma_0}\d Z\,\Res\left(\frac{W^ne^{\tau\left(\sqrt2r-2W\right)^2-\sqrt2rW}}{Z^{m+1}e^{\tau\left(\sqrt2r-2Z\right)^2-\sqrt2rZ}};W=Z\right)\\
&\qquad=\frac1{2\pi\I}\oint_{\Gamma_0}\d Z\,Z^{n-m-1}=\id(n,m)
\end{aligned}\end{equation}
which shows \eqref{fugu}.

In the same way, \eqref{fug-u} follows by
\begin{equation}\begin{aligned}
&\frac2{(2\pi\I)^2}\int_\R\d u\int_{\I\R}\d W\oint_{\Gamma_0}\d Z\,\frac{W^ne^{\tau\left(\sqrt2r-2W\right)^2-\sqrt2rW}}{Z^{m+1}e^{\tau\left(\sqrt2r-2Z\right)^2-\sqrt2rZ}}f_W(u)g_Z(-u)\\
&\qquad=\frac1{2\pi\I}\oint_{\Gamma_0}\d Z\,\Res\left(\frac{W^ne^{\tau\left(\sqrt2r-2W\right)^2-\sqrt2rW}}{Z^{m+1}e^{\tau\left(\sqrt2r-2Z\right)^2-\sqrt2rZ}};W=\sqrt2r-Z\right)\\
&\qquad=K_0(n,m).
\end{aligned}\end{equation}
This completes the proof of the lemma.
\end{proof}

\begin{proof}[Proof of Proposition~\ref{prop:derivative}]
First observe that due to definition \eqref{defK0hat},
\begin{equation}\label{K0derivative}
\frac\partial{\partial R}\wh K_0(\xi,\zeta)=\left(\frac\partial{\partial\xi}+\frac\partial{\partial\zeta}\right)\wh K_0(\xi,\zeta).
\end{equation}
By writing the resolvent of $\wh K_0$ as a Neumann series and by applying \eqref{K0derivative} to each term of the series, one obtains
\begin{equation}\label{resolventderivative}
\frac\partial{\partial R}(\id-\wh K_0)^{-1}(\xi,\zeta)=\left(\frac\partial{\partial\xi}+\frac\partial{\partial\zeta}\right)(\id-\wh K_0)^{-1}(\xi,\zeta)
-(\id-\wh K_0)^{-1}\wh K_0(\xi,0)(\id-\wh K_0)^{-1}\wh K_0(0,\zeta).
\end{equation}
Further, by \eqref{defPhihat},
\begin{equation}\label{phipsiderivative}
\frac\partial{\partial R}\wh\Psi_T^\xi(U)=\frac\partial{\partial\xi}\wh\Psi_T^\xi(U)\qquad\mbox{and}\qquad
\frac\partial{\partial R}\wh\Phi_T^\zeta(U)=\frac\partial{\partial\zeta}\wh\Psi_T^\zeta(U).
\end{equation}
Now one can take the derivative of the kernel $\wh K^\ext$ in \eqref{defKexthat} with respect to $R$.
Using \eqref{resolventderivative} and \eqref{phipsiderivative}, the proposition follows by direct computation.
\end{proof}

\begin{proof}[Proof of Lemma~\ref{lemma:NRextend}]
For a function $h$ of the form \eqref{defhmp}, one can define approximating functions $h_\varepsilon\in H^1([0,1])$ for any small $\varepsilon>0$ such that
as $\varepsilon$ decreases to $0$, the functions increasingly approach $h$.
With other words, $h_\varepsilon(x)\to h(x)$ increasingly as $\varepsilon\to0$ for any $x\in[0,1]$.

Then the events $E_\varepsilon=\{B_N(t)<h_\varepsilon(t)\mbox{ for }t\in[0,1]\}$ increase to $E_0=\{B_N(t)<h(t)\mbox{ for }t\in[0,1]\}$ as $\varepsilon\to0$,
hence $\P(E_\varepsilon)\to\P(E_0)$ by the continuity of the measure.
Similarly, the events $\wt E_\varepsilon=\{\wt b(\tau)\le\wt h_\varepsilon(\tau)\mbox{ for }\tau\in(\tau_1,\tau_2)\}$ which appear in \eqref{defT1} used in the definition \eqref{defKN} of $K_N^h$
increase to $\wt E_0=\{\wt b(\tau)\le\wt h(\tau)\mbox{ for }\tau\in(\tau_1,\tau_2)\}$ as $\varepsilon\to0$,
since the functions $\wt h_\varepsilon$ increase to $\wt h$ pointwise as $\varepsilon\to0$, see \eqref{deftauhtilde}.
Hence $\P(\wt E_\varepsilon)\to\P(\wt E_0)$.

To complete the proof, the convergence of the corresponding Fredholm determinants on the right-hand side of \eqref{NBBbelowh} has to be shown.
From $\P(\wt E_\varepsilon)\to\P(\wt E_0)$, one has the pointwise convergence of the operators in the Fredholm determinant.
On the other hand, by Lemma~\ref{lemma:traceclass}, $K_N^h$ is a trace class operator for any function $h$, i.e.\ the Fredholm determinant series converges absolutely,
hence the corresponding Fredholm determinants on the right-hand side of \eqref{NBBbelowh} converge as $\varepsilon\to0$ by dominated convergence.
\end{proof}

\begin{proof}[Proof of Lemma~\ref{lemma:error}]
This proof follows the lines of the proof of Lemma~2.3 in~\cite{NR15}.
We first rewrite the operator $e^{-2LD}-R_{-L,L}$ as follows.
We substitute \eqref{ABL} into the definitions \eqref{freeev} and \eqref{defR} and we use the identity
\begin{equation}
-\frac{(e^Ly-e^{-L}x)^2}{e^{2L}-e^{-2L}}+\frac{(e^Ly-e^{-L}x-(e^{2L}-e^{-2L})\frac r{\sqrt2})^2}{e^{2L}-e^{-2L}}=\frac{(e^{2L}-e^{-2L})r^2}2-\sqrt2r(e^Ly-e^{-L}x)
\end{equation}
to simplify the exponential factors.
Then one has the decomposition
\begin{equation}\label{opproduct}
e^{-2LD}-R_{-L,L}=\Gamma_1\Gamma_2\Gamma_3
\end{equation}
where
\begin{align}
\Gamma_1(x,u_1)&=e^{-x^2/2+\sqrt2e^{-L}xr+u_1^2/8\tau_1}T_{e^{-2L}/4,\tau_1}\left(e^{-L}x-\frac{(1+e^{-2L})r}{\sqrt2},u_1\right)\id_{u_1\le H_1},\label{defGamma1}\\
\Gamma_2(u_1,u_2)&=e^{-u_1^2/(8\tau_1)+u_2^2/8\tau_k}T_{\tau_1,\tau_k}^{\tau_i,H_i}(u_1,u_2),\\
\Gamma_3(u_2,y)&=\id_{u_2\le H_k}T_{\tau_k,e^{2L}/4}\left(u_2,e^Ly-\frac{(1+e^{2L})r}{\sqrt2}\right)e^{-u_2^2/8\tau_k+y^2/2+L-\sqrt2e^Lyr+(e^{2L}-e^{-2L})r^2/2}.\label{defGamma3}
\end{align}
The extra conjugation by $e^{u_1^2/8\tau_1}$ and by $e^{u_2^2/8\tau_k}$ was introduced because in this way all the operators $\Gamma_1,\Gamma_2,\Gamma_3$ have finite norm as shown below.
Next we decompose the error term as $\Omega_L=\Omega_L^1+\Omega_L^2$ with
\begin{align}
\Omega_L^1&=P_{\sqrt2 r\cosh L}(e^{-2LD}-R_{-L,L})\ol P_{\sqrt2 r\cosh L},\label{defOmega1}\\
\Omega_L^2&=(e^{-2LD}-R_{-L,L})P_{\sqrt2 r\cosh L}.\label{defOmega2}
\end{align}
We bound the trace norm of
\begin{equation}
\wt\Omega_L=e^{LD}K_{\Herm,N}\Omega_L^1e^{LD}K_{\Herm,N}+e^{LD}K_{\Herm,N}\Omega_L^2e^{LD}K_{\Herm,N}
\end{equation}
as follows.
One has by \eqref{opproduct} and \eqref{defOmega1} that
\begin{multline}\label{Omega1decomp}
\|e^{LD}K_{\Herm,N}\Omega_L^1e^{LD}K_{\Herm,N}\|_1\\
\le\|e^{LD}K_{\Herm,N}P_{\sqrt2 r\cosh L}\Gamma_1\|_2\,\,\|\Gamma_2\|_{\op}\,\,\|\Gamma_3\ol P_{\sqrt2 r\cosh L}e^{LD}K_{\Herm,N}\|_2.
\end{multline}
By definition, one can write the square of the first Hilbert--Schmidt norm as
\begin{equation}\label{Gamma1sqnorm}\begin{aligned}
&\|e^{LD}K_{\Herm,N}P_{\sqrt2 r\cosh L}\Gamma_1\|_2^2\\
&\qquad=\sum_{n,m=0}^{N-1}\int_\R\d x\int_{-\infty}^{H_1}\d y\int_{\sqrt2r\cosh L}^\infty\d w\int_{\sqrt2r\cosh L}^\infty\d z\\
&\qquad\qquad\times e^{L(n+m)}\varphi_n(x)\varphi_m(x)\varphi_n(w)\varphi_m(z)\Gamma_1(w,y)\Gamma_1(z,y)\\
&\qquad=\sum_{n=0}^{N-1}e^{2nL}\int_{-\infty}^{H_1}\d y\left(\int_{\sqrt2r\cosh L}^\infty\d z\,\varphi_n(z)\Gamma_1(z,y)\right)^2\\
&\qquad\le Ne^{2(N-1)L}\int_{-\infty}^{H_1}\d y\left(\int_{\sqrt2r\cosh L}^\infty\d z\,\varphi_n(z)^2\right)\left(\int_{\sqrt2r\cosh L}^\infty\d z\,\Gamma_1(z,y)^2\right)\\
&\qquad\le Ne^{2(N-1)L}\int_{\sqrt2r\cosh L}^\infty\d z\int_\R\d y\,\Gamma_1(z,y)^2
\end{aligned}\end{equation}
where we used first that the harmonic oscillator functions $\varphi_n$ are orthonormal, then the Cauchy--Schwarz inequality, and finally the orthonormal property of $\varphi_n$ again.
In the definition of $\Gamma_1$ \eqref{defGamma1} and by comparing it with \eqref{defT2}, one can give the upper bound
\begin{equation}\label{Gamma1int}\begin{aligned}
\int_\R\d y\,\Gamma_1(z,y)^2
&\le e^{-z^2+2\sqrt2e^{-L}zr}\int_\R\d y\,e^{\frac{y^2}{4\tau_1}}\phi_{2\tau_1-e^{-2L}/2}\left(y-e^{-L}z+\frac{1+e^{-2L}}{\sqrt2}r\right)^2\\
&=e^{-z^2+2\sqrt2e^{-L}zr}\int_\R\d y\,\frac{e^{\frac{y^2}{4\tau_1}}}{\pi(4\tau_1-e^{-2L})}\exp\left(-\frac{(y-e^{-L}z+\frac{1+e^{-2L}}{\sqrt2}r)^2}{2\tau_1-e^{-2L}/2}\right)\\
&=e^{-(1+o(1))z^2+o(1)z}\int_\R\d y\frac1{4\pi\tau_1}e^{-(1+o(1))\frac{y^2}{4\tau_1}+(1+o(1))\frac{ry}{\sqrt2\tau_1}-\frac{r^2}{4\tau_1}+o(1)yz+o(1)}\\
&=\frac1{\sqrt{4\pi\tau_1}}e^{-(1+o(1))z^2+\frac{r^2}{4\tau_1}+o(1)z+o(1)}
\end{aligned}\end{equation}
by computing the Gaussian integral in the last step.
The $o(1)$ above means a term which does neither depend on $y$ nor $z$ and which goes to $0$ as $L\to\infty$.
Putting \eqref{Gamma1sqnorm} and \eqref{Gamma1int} together, one obtains
\begin{equation}\label{Gamma1sqnorm2}\begin{aligned}
\|e^{LD}K_{\Herm,N}P_{\sqrt2 r\cosh L}\Gamma_1\|_2^2&\le\frac{Ne^{2(N-1)L+\frac{r^2}{4\tau_1}}}{\sqrt{4\pi\tau_1}}\int_{\sqrt2r\cosh L}^\infty\d z\,e^{-(1+o(1))z^2+o(1)z+o(1)}\\
&\le\frac{Ne^{2NL+\frac{r^2}{4\tau_1}}}{\sqrt{4\pi\tau_1}}e^{-2r^2(\cosh L)^2(1+o(1))}\\
&\le c_1e^{2NL-c_2e^{2L}}
\end{aligned}\end{equation}
with positive constants $c_1$ and $c_2$ for $L$ large enough.
We used the Chernoff bound on the tail of the normal distribution in the second inequality.

Obtaining a bound on $\|\Gamma_3\ol P_{\sqrt2 r\cosh L}e^{LD}K_{\Herm,N}\|_2^2$ is very similar.
There is a difference in the step which corresponds to \eqref{Gamma1int}.
It can be done as follows.
\begin{equation}\label{Gamma3int}\begin{aligned}
\int_\R\d x\,\Gamma_3(x,z)^2&\le e^{z^2+2L-2\sqrt2e^Lzr+(e^{2L}-e^{-2L})r^2}\!\!\int_\R\d x\,e^{-\frac{x^2}{4\tau_k}}\phi_{e^{2L}/2-2\tau_k}\Big(x-e^Lz+\frac{(1+e^{2L})r}{\sqrt2}\Big)^2\\
&=e^{2L-2r^2-(1+o(1))z^2+o(1)z}\int_\R\d x\,\frac1{\sqrt{\pi e^{2L}}}e^{-(1+o(1))\frac{x^2}{4\tau_k}-(1+o(1))2\sqrt2xr+o(1)xz+o(1)}\\
&=\frac1{\pi e^{2L}}e^{2L+(8\tau_k-2)r^2-(1+o(1))z^2+o(1)z+o(1)}
\end{aligned}\end{equation}
where the $o(1)$ term are again independent of $y$ and $z$ and they go to $0$ as $L\to\infty$.
The computation \eqref{Gamma3int} results in a bound
\begin{equation}\label{Gamma3sqnorm}
\|\Gamma_3\ol P_{\sqrt2 r\cosh L}e^{LD}K_{\Herm,N}\|_2\le c_1e^{NL}
\end{equation}
very similarly as in \eqref{Gamma1sqnorm2}.
The factor $e^{-c_2e^{2L}}$ is not present due to the fact that the projection $P_{\sqrt2r\cosh L}$ is replaced by $\ol P_{\sqrt2r\cosh L}$.

Finally, the operator norm of $\Gamma_2$ can be bounded in the following way.
\begin{equation}\label{Gamma2sqnorm}\begin{aligned}
\|\Gamma_2\|_{\op}^2
&\le\sup_{y\in\R}\int_\R\d x\,\Gamma_2(x,y)^2\le\sup_{y\in\R}\int_\R\d x\,e^{-\frac{x^2}{4\tau_1}+\frac{y^2}{4\tau_k}}\phi_{2(\tau_k-\tau_1)}(y-x)^2\\
&=\sup_{y\in\R}\frac12\sqrt{\frac{\tau_1}{\pi(\tau_k^2-\tau_1^2)}}e^{-\frac{(\tau_k-\tau_1)y^2}{4\tau_k(\tau_1+\tau_k)}}=\frac12\sqrt{\frac{\tau_1}{\pi(\tau_k^2-\tau_1^2)}}
\end{aligned}\end{equation}
by straightforward computation involving a Gaussian integral.
Putting \eqref{Omega1decomp}, \eqref{Gamma1sqnorm}, \eqref{Gamma2sqnorm} and \eqref{Gamma3sqnorm} together proves that the error corresponding to $\Omega_L^1$ goes to $0$ as $L\to\infty$.
The proof for $\Omega_L^2$ can be done similarly.
\end{proof}


\end{document}